\documentclass[a4paper, 11pt]{article}

\usepackage{srcltx,graphicx,epstopdf}
\usepackage{amsmath, amssymb, amsthm}
\usepackage{color}
\usepackage{xcolor}
\usepackage{multirow}
\usepackage{subfigure} 
\usepackage{hyperref}
\usepackage[margin=1in]{geometry}

\usepackage[numbers,sort&compress]{natbib}

\usepackage{verbatim}  
\usepackage{bm}
\usepackage{amsmath, amssymb, amsthm} 
\numberwithin{equation}{section}  

\newtheorem{theorem}{Theorem}
\newtheorem{lemma}{Lemma}
\newtheorem{remark}{Remark}

\newcommand{\bx}{\bm{x}}
\newcommand{\bu}{\bm{u}}
\newcommand{\bv}{\bm{v}}
\newcommand{\bE}{\bm{E}}
\newcommand{\bB}{\bm{B}}
\newcommand{\bq}{\bm{q}}
\newcommand{\bJ}{\bm{J}}

\newcommand{\mT}{\mathcal{T}}
\newcommand{\mP}{\mathcal{P}}
\newcommand{\He}{{H}}
\newcommand{\bxi}{\bm{\xi}}
\newcommand{\bbR}{\mathbb{R}}

\newcommand{\bz}{\bm{0}}
\newcommand{\dd}{\,\mathrm{d}}
\newcommand{\bbN}{\mathbb{N}}
\newcommand{\bj}{\bm{j}}

\newcommand{\mE}{\mathcal{E}}
\newcommand{\mV}{\mathcal{V}}
\newcommand{\mEt}{\mE_{\rm total}}

\newcommand\pd[2]{\dfrac{\partial {#1}}{\partial {#2}}}

\newcommand{\bbf}{\bm{f}}

\numberwithin{theorem}{section}
\numberwithin{table}{section}

\title{Highly efficient energy-conserving moment method for the multi-dimensional Vlasov-Maxwell system}
\author{Tianai Yin\thanks{Beijing Computational Science Research Center, Beijing,
     100193, China. Email: {\tt tianai.yin@csrc.ac.cn}}, ~~Xinghui Zhong\thanks{School of Mathematical Sciences, Zhejiang University, Hangzhou, 310027, China. Email: {\tt zhongxh@zju.edu.cn}.},
    ~~Yanli
  Wang\thanks{Beijing Computational Science Research Center, Beijing,
  100193,    China. Email: {\tt ylwang@csrc.ac.cn}.} }

\begin{document}
\maketitle
\begin{abstract}
  In this paper, we propose an energy-conserving numerical method to
  solve the Vlasov-Maxwell (VM) system based on the regularized moment
  method proposed in \cite{2014Globally}. The globally hyperbolic
  moment system is deduced for the multi-dimensional VM system under
  the framework of Hermite expansions, where the expansion center and
  the scaling factor are set as the macroscopic velocity and local
  temperature, respectively. Thus, the effect of the Lorentz force
  term can be reduced into several ODEs regarding the macroscopic
  velocity and the higher-order moment coefficients, which can
  significantly reduce the computational cost of the whole system. An
  energy-conserving numerical scheme is proposed to solve the moment
  equations and Maxwell's equations, where only a small linear
  equation system needs to be solved for the implicit part. Benchmark
  examples such as the Landau damping, two-stream instability, Weibel
  instability, and the two-dimensional Orszag-Tang vortex problem are
  studied to validate the efficiency and excellent energy-preserving
  property of the numerical scheme.
  

\end{abstract}

\section{Introduction}

 Plasma, which exists widely in the universe, is the fourth fundamental state of matter after solid, liquid, and gas. Understanding the complex behavior of plasma has led to significant advances ranging from space physics and fusion energy, to high-power microwave generation and large-scale particle accelerators. One of the fundamental models in plasma physics is the Vlasov system \cite{Plasma}, which describes the time evolution of the distribution function of collisionless charged particles with long-range interactions. The long-range interactions may occur under a self-generated electromagnetic field. For example, the evolution of the electromagnetic field can be modeled by Maxwell's equations or Poisson's equation in the zero-magnetic field limit, resulting in the well-known Vlasov-Maxwell (VM) or Vlasov-Poisson (VP) systems. 




 Numerically solving the VM system is a difficult task. There are
 several challenges such as the high dimensionality, the conservation
 of physical quantities due to the Hamiltonian structure of the system
 \cite{morrison1980maxwell, morrison2013general}, various physical
 phenomena, nonlinearity, etc. Generally, there are two types of
 methods, i.e. stochastic methods and deterministic methods. Among the
 stochastic methods, the particle-in-cell (PIC) method
 \cite{dawson1962one, tskhakaya2007particle, birdsall2018plasma} has
 been a prevalent numerical tool for a long time. In the PIC method,
 each plasma particle is considered, and all macro-quantities are
 computed from the position and velocity of these
 particles. Meanwhile, the force acting on the particles is computed
 from the field equations. The PIC method enjoys the advantage of a
 relatively low cost for high-dimensional problems and the accurate
 computation of the convection term due to the Lagrangian nature,
 while suffering from the statistical noise and low resolution of the
 electron distribution function, especially when dealing with low
 temperature and high densities plasma \cite{duclous2009high}. In
 addition, several works have been done for the PIC method to preserve
 physical conservation. The method in \cite{birdsall2018plasma} can
 conserve mass and momentum in a uniform computational grid. In
 \cite{lewis1970energy, evstatiev2013variational}, the time-explicit
 energy-conserving schemes were presented. Moreover, the
 energy-conserving implicit PIC algorithms for the VM system with a
 finite difference method in the physical space were proposed in
 \cite{markidis2011energy}, while a fully implicit solver for the VA
 system based on Newton-Krylov methods was proposed in
 \cite{chen2011energy}. Recently, highly efficient unified gas-kinetic
 wave-partical (UGKWP) methods, combining the deterministic UGKS
 method and the stochastic Monte Carlo method, has been proposed to
 solve the Boltzmann equations and extended to the VM system in
 \cite{UGKWPVM}.

 For the deterministic methods, there are several different types of
 solvers. For example, the Runge-Kutta discontinuous Galerkin methods
 (RKDG) have been proposed in \cite{cheng2014RKDG} for the VP and VM
 system. Semi-Lagrangian methods \cite{califano2001fast,
   califano1998kinetic, qiu2010conservative, qiu2011positivity}, which
 are also combined with WENO reconstruction \cite{Qiu2021Con} or DG
 methods \cite{Qiu2021High}, are adopted to solve Vlasov-type
 equations.  Moreover, the spectral methods, where the Fourier basis
 function or the Hermite polynomials are utilized to approximate the
 distribution functions, have been widely studied. We refer readers to
 \cite{Multi_hermite, eliasson2003numerical} and the references cited
 therein for more details. Besides, the finite difference method
 \cite{umeda2009two} and Hamiltonian splitting method \cite{Hamilton}
 both have specific solvers for the VM system. For most of the methods
 proposed above, the conservation of the number of particles is easy
 to achieve, while the conservation of the total energy will be much
 more difficult. Achieving energy conservation is also a developing
 goal of the above methods. The implicit and explicit
 energy-preserving RKDG methods were proposed in \cite{VA_Zhong,
   VM_Zhong, cheng_numerical_2015} for the Vlasov-Amp\`{e}re (VA) and
 VM systems, respectively.  The conservation of charge, momentum and
 energy is discussed for the finite volume scheme with moment-based
 acceleration algorithm proposed in \cite{TAITANO2015718,
   taitano2015charge} for the VA system. The energy conserving
 discontinuous Galerkin spectral element method and Legendre-Fourier
 spectral method for the VP system are discussed in
 \cite{madaule2014energy, manzini2016legendre},
 respectively. Moreover, a conservative Hermite spectral method
 combined with discontinuous Galerkin method is proposed in
 \cite{FilbetXiong2022} for the VP system, with its stability further
 studied in \cite{bessemoulin2022stability}.  In
 \cite{einkemmer2021mass}, a dynamical low-rank algorithm that
 conserves mann, momentum, and energy, is proposed for the VP system.
 Besides, the conservation of the total energy is achieved when the
 central numerical flux is employed with the Maxwell's equations in
 the framework of the Hermite-DG method for the VM system
 \cite{hermite_DG_energy}.  A numerical scheme preserves both
 positivity of the distribution function and total energy is proposed
 for the VM system in \cite{allmann2022energy}.

 Recently, a regularized moment method was presented in
 \cite{2014Globally} for the kinetic equation, where series of Hermite
 expansions are adopted to approximate the distribution function,
 where the expansion center and the scaling factor of the basis
 functions are set as the macroscopic velocity and the local
 temperature, respectively, which are chosen adaptively according to
 the distribution function. Moreover, a special regularization method
 was proposed to obtain the globally hyperbolic moment systems for
 kinetic equations \cite{2014Globally}. This regularized moment method
 naturally bridges between the macroscopic and microscopic
 descriptions of the particles. It has been verified that this method
 has spectral convergence with the number of moments, and has been
 successfully applied to solve the Boltzmann-type equations
 \cite{2014Globally}, the VP system \cite{Vlasov2012}, and the 1D VP
 Fokker-Planck equations numerically \cite{wang2017VPFP}.

Though it is verified to be efficient, the numerical schemes adopted in \cite{Vlasov2012} can only conserve the  mass and momentum for the VP system, and no results on the conservation of the total energy have been achieved. Inspired by the efficiency of the regularized moment method and the energy-conserving RKDG method in \cite{VM_Zhong, VA_Zhong}, we propose an energy-conserving regularized moment method for the multi-dimensional VM system. The distribution function is expanded by a series of Hermite functions, with the expansion center and the scaling factor chosen adaptively as in \cite{Vlasov2012}. With this specially chosen expansion center, the effect of the Lorentz force from the electromagnetic field can be changed into a linear combination of the moment coefficients, in which case the computational cost can be significantly reduced. To design the energy-conserving scheme, the moment system is split into the convection step and the Lorentz force step by the Strang-splitting method. Thus, the effect of the Lorentz force term is reduced to several ODEs regarding the macroscopic velocity and high-order moment coefficients. Most importantly, an implicit scheme is developed to solve Maxwell's equations and Lorentz force step simultaneously, enabling the  conservation of the mass and total energy for the VM system. Only a  small system of linear equations needs to be solved for this implicit scheme. Several numerical examples of the VM system, such as the one-dimensional Landau damping, two-stream instability, Weibel instability, and the two-dimensional Orszag-Tang vortex problems, are studied to exhibit the high efficiency of this energy-conserving numerical method.

The rest of the paper is as follows: the VM system and related physical properties are briefly introduced  in Sec. \ref{sec:FPL}. The regularized moment method and the deduction of the globally hyperbolic moment equations are proposed in Sec. \ref{sec:method}. In Sec. \ref{sec:scheme}, the temporal and fully discrete energy-conserving numerical schemes are presented with the related proof. Several numerical tests are studied in Sec. \ref{sec:num}, with some concluding remarks in Sec. \ref{sec:conclusion}.

\section{The Vlasov-Maxwell equations}
\label{sec:FPL}
In this section, we introduce the VM system. Under suitable scaling of the characteristic time, the length, and the characteristic electric and magnetic field, the dimensionless Vlasov equation is given by
\begin{equation}
		\frac{\partial f}{\partial t}+\bv \cdot \nabla_{\bm{x}} f + (\bm{E}+\bm{v}\times\bm{B})\cdot \nabla_{\bm{v}} f=0,\qquad (t,\bm{x},\bm{v})\in \mathbb{R}^+\times\Omega \times \mathbb{R}^3,
\label{eq:vlasov}
\end{equation}
where $f(t,\bm{x},\bm{v})$ is the distribution function describing the motion of the charged particles at position $\bm{x}$ with microscopic velocity $\bm{v}$ at time $t$. $\Omega \subset \bbR^3$ denotes the spatial domain. The  $3$-dimensional microscopic velocity space is set as $\mathbb{R}^3$. 
$(\bm{E}(t,\bx),\bm{B}(t,\bx))$ is the electromagnetic field, when it is modeled by the dimensionless Maxwell's equations
\begin{equation}
\label{eq:maxwell}
		\left\{
		\begin{aligned}
			&\frac{\partial\bm{E}}{\partial t}-\nabla_{\bm{x}}\times \bm{B}=-\bm{J},\\
			&\frac{\partial \bm{B}}{\partial t}+\nabla_{\bm{x}}\times\bm{E}=0,
		\end{aligned}
		\right.
\end{equation}
 with the current density $\bJ$ defined as 
\begin{equation}
\label{eq:current}
	\bm{J}(t,\bm{x})=\int_{\mathbb{R}^3}\bm{v}f(t,\bm{x},\bm{v}) \dd\bm{v},
\end{equation}
we obtain the Vlasov-Maxwell system.
It is worth mentioning that Maxwell's equations \eqref{eq:maxwell} are further supplemented by Gauss's law \cite{jackson1977classical} as
\begin{equation}
\label{eq:gauss_law}
	\nabla_{\bx}\cdot\bm{E}= \rho - \rho_{\rm bound}, \qquad\nabla_{\bx}\cdot\bm{B}=0,
\end{equation}
with the density $\rho$ defined as 
\begin{equation}
  \label{eq:rho}
  \rho = \int_{\bbR^3} f(t,\bx,\bv) \dd \bx,
\end{equation}
and $\rho_{\rm bound}$ being the density of particles from the background. When the background is vacuum, $\rho_{\rm bound} = 0$. In the zero-magnetic limit, the VM system becomes the VP or VA system 
\begin{equation}
\label{eq:vp}
\begin{aligned}
& & \frac{\partial f}{\partial t}+\bv \cdot \nabla_{\bm{x}} \cdot f+ \bm{E}\cdot \nabla_{\bm{v}} f=0, & \\
 & &\bm{E}=-\nabla \phi,\qquad -\Delta\phi=
 \rho - \rho_{\rm{bound}}, &
 \qquad \text{Vlasov-Poisson}, \\
 && \frac{\partial \bm{E}}{\partial t}=-\bm{J}, & \qquad \text{Vlasov-Amp\`{e}re}.
\end{aligned}
\end{equation}
The VP and VA systems are equivalent in the absence of external fields, when the charge continuity equation
\begin{align*}
  \pd{\rho}{t} +\nabla_{\bx} \cdot {\bm J}=0
\end{align*}
is satisfied.

In this paper, we focus on the VM system. All the discussions can be extended to the VA system by properly adjusting $\bm B=0$. Besides, the density $\rho$ defined in \eqref{eq:rho}, we are also interested in the physical variables such as the macroscopic velocity $\bu$ and the thermal temperature $\mT$ satisfying
\begin{align}
\label{eq:macro}
	\rho\bm{u}=\int_{\mathbb{R}^3}\bm{v}f\dd\bm{v},\qquad 
  \frac{1}{2}\rho|\bm{u}|^2+\frac{3}{2}\rho\mathcal{T}= \frac{1}{2}\int_{\mathbb{R}^3}|\bm{v}|^2 f\dd\bm{v},
\end{align}
as well as the heat flux $\bq$ and the pressure tensor $p_{ij}$ defined as 
\begin{align}
  \label{eq:heatflux_pressure} 
  	\bm{q}=\frac{1}{2}\int_{\mathbb{R}^3}|\bm{v}-\bm{u}|^2(\bm{v}-\bm{u}) f \dd\bm{v},\qquad 
  	p_{i j}=\int_{\mathbb{R}^3}(v_i-u_i)(v_j-u_j)f\dd\bm{v}, \quad i,j = 1, 2, 3.
\end{align}
More details can be found in \cite{struchtrup2005macroscopic}. 

For a collision-based system composed with a large number of non-interacting particles, where quantum effects can be ignored, the most likely distribution function \cite{krall1973principles} is the Maxwellian distribution, which is also known as the equilibrium distribution, defined as
\begin{equation}
	f_{\rm eq}(t,\bm{x},\bm{v})=\frac{\rho(t,\bm{x})}{[2\pi\mathcal{T}(t,\bm{x})]^{3/2}}
	\text{exp}\left(- \frac{\left(\bm{v}-\bm{u}(t,\bm{x}) \right)^2}{2\mathcal{T}(t,\bm{x})}\right).
\end{equation}
Moreover, the VM system conserves the total particle number $\mathcal{P}(t)$ and the total energy $\mE_{\rm total}(t)$, defined as
\begin{subequations}
\begin{align}
   \label{eq:total_mass}
   \mathcal{P}(t)= \int_{\Omega\times\mathbb{R}^3}f(t,\bm{x},\bm{v})
 \dd\bm{x}\dd\bm{v}, \qquad 
  \mathcal{E}_{\rm total}(t)= \mathcal{E}_K(t)+\mathcal{E}_B(t)+\mathcal{E}_E(t), 
\end{align}
\end{subequations}
where $\mE_{\rm total}(t)$ is composed of the kinetic and electromagnetic energy as 
\begin{equation}
\label{eq:energy}
\begin{gathered}
		\mathcal{E}_K(t)=\frac{1}{2}\int_{\Omega}\int_{\mathbb{R}^3}|\bm{v}|^2f(t,\bm{x},\bm{v}) \dd\bm{v}\dd\bm{x} = \frac{1}{2}\int_{\Omega} \rho|\bu|^2 + 3 \rho \mT \dd \bx, \\
		\mathcal{E}_E(t)=\frac{1}{2}\int_{\Omega}|\bm{E}(t,\bm{x}) |^2 \dd\bm{x},\qquad 
		\mathcal{E}_B(t)=\frac{1}{2}\int_{\Omega}|\bm{B}(t,\bm{x}) |^2 \dd\bm{x}.
\end{gathered}
\end{equation}
\section{The moment method}
\label{sec:method}
In this section, we lay out the details of the moment method for the VM system, to obtain the moment equations of the VM system and discuss their hyperbolicity property. 

\subsection{Series expansion and moment systems}
Following the method in \cite{Vlasov2012}, the distribution function $f(t, \bx, \bv)$ is expanded into  Hermite series as 
\begin{equation}
    \label{eq:expansion} 
	f(t,\bm{x},\bm{v})=\sum_{\alpha\in\mathbb{N}^3}f_{\alpha}(t,\bm{x})
	\mathcal{H}_{\mathcal{T},\alpha}
	\left(\frac{\bm{v}-\bm{u}(t,\bm{x})}{\sqrt{\mathcal{T}(t,\bm{x})}}\right),
\end{equation}
where the basis functions $\mathcal{H}_{\mathcal{T},\alpha}(\cdot)$ are defined as
\begin{equation}
\label{eq:basic_fun}
	\mathcal{H}_{\mathcal{T},\alpha}(\bm{\xi})=\frac{1}{(2\pi)^{3/2}}\mathcal{T}^{-\frac{|\alpha|+3}{2}}\prod_{d=1}^{3}
\He_{\alpha_d}(\xi_d)\text{exp}\left(-\frac{\xi_d^2}{2}\right ).
\end{equation}
Here $\alpha=(\alpha_1, \alpha_2, \alpha_3)$ is a $3$-dimensional multi-index,   $\bm{\xi}=(\xi_1,\xi_2,\xi_3)$ is defined as 
\begin{equation}
\label{eq:xi}
	\bxi=\frac{  \bm{v}-\bm{u}(t,\bm{x}) } {\sqrt{\mathcal{T}(t,\bm{x}) }},
\end{equation}
and $\He_{\alpha_d}$ is the Hermite polynomial
\begin{equation}
\label{eq:Her}
	\He_{k}(x)=(-1)^k \text{exp}\left(\frac{x^2}{2}\right)\frac{\dd^k}{\dd x^k}\text{exp}\left(-\frac{x^2}{2}\right). 
\end{equation}
For convenience, $\He_k(x)$ is taken as zero if $k<0$. Thus, $\mathcal{H}_{\mathcal{T},\alpha}$ is zero when any component of $\alpha$ is negative. Hermite polynomials have several important properties, which are  useful for deriving the moment equations of the VM system, such as 
\begin{itemize}
    \item Orthogonality:
\begin{equation} 
    \label{eq:orth}
\int_{\bbR}\He_{m}(x)\He_k(x)\text{exp}\left(\frac{-x^2}{2}\right) \dd x
	=m!\sqrt{2\pi}\delta_{mk},
\end{equation}	
\item Recursion relation:
\begin{equation} 
\label{eq:rec}
\He_{k+1}(x)=x \He_k(x)-k \He_{k-1}(x),
\end{equation}
\item Differential relation:
\begin{equation}
\label{eq:diff}
\quad \He_k'(x)=k \He_{k-1}(x).
\end{equation}
\end{itemize}
With the orthogonality of the Hermite polynomials, the moment coefficients $f_{\alpha}$ in \eqref{eq:expansion} satisfy
\begin{gather}
\label{eq:moment_macr}
	\rho=f_{\bz},\qquad
	f_{e_i}=0,\quad i=1,2,3,\qquad\sum_{d=1}^3 f_{2e_d}=0, \\
	q_i=2f_{3e_i}+\sum_{d=1}^3 f_{2e_d+e_i},\qquad
	p_{ij}=\delta_{ij}\rho\mathcal{T}+(1+\delta_{ij})f_{e_i+e_j},\quad i,j=1,2,3,
\end{gather}	
where $e_d$ is the $3$-dimensional multi-index whose $d$-th entry is $1$ and all other entries are zero.
Moreover, based on the properties of the Hermite polynomials, it holds for the basis functions 
$\mathcal{H}_{\mathcal{T},\alpha}$ that 
\begin{align}
\label{eq:relation_basfun}
\frac{\partial}{\partial v_j}\mathcal{H}_{\mathcal{T},\alpha}\left(\frac{\bm{v}-\bm{u}}{\sqrt{\mathcal{T}}}\right) & =-\mathcal{H}_{\mathcal{T},\alpha+e_j}\left(\frac{\bm{v}-\bm{u}}{\sqrt{\mathcal{T}}}\right), \\
v_j\mathcal{H}_{\mathcal{T},\alpha}\left(\frac{\bm{v}-\bm{u}}{\sqrt{\mathcal{T}}}\right)& =
	\mathcal{T}\mathcal{H}_{\mathcal{T},\alpha+e_j}
	+u_j\mathcal{H}_{\mathcal{T},\alpha}
	+\alpha_j \mathcal{H}_{\mathcal{T},\alpha-e_j}, \qquad j=1,2,3.	
\end{align}
Thus, the force term $(\bm{E}+\bm{v}\times\bm{B})\cdot \nabla_{\bm{v}}f$ is expanded  as
\begin{equation}
\begin{split}
	&(\bm{E}+\bm{v}\times\bm{B})\cdot \nabla_{\bm{v}}f\\
	=&-\sum_{\alpha \in \bbN^3}\left\{\sum_{d,l,m=1}^{3}\Big[
	\left(E_d+\epsilon_{dlm}u_l B_m\right) f_{\alpha-e_d}
	+\epsilon_{dlm}(\alpha_l+1)B_m f_{\alpha-e_d+e_l} \Big]\right\}
	\mathcal{H}_{\mathcal{T},\alpha}\left(\frac{\bm{v}-\bm{u}}{\sqrt{\mathcal{T}} }	\right),
 \end{split}
 \label{eq:Lorentz_expansion}
\end{equation}
where $\epsilon_{d l m}$ are the Levi-Civita symbols defined as
\begin{equation}
\label{eq:LC}
\epsilon_{d l m}=
	\left\{
		\begin{aligned}
			1,&\quad \textnormal{ if } (d,l,m) \textnormal{ is a cyclic permutation of}\;(1,2,3),\\
			-1,&\quad \textnormal{ if } (d,l,m) \textnormal{ is an anticyclic permutation of}\;(1,2,3),\\
			0,&\quad \textnormal{ if } (d-l)(l-m)(m-d)=0.
		\end{aligned}
	 \right.
\end{equation}

Substituting \eqref{eq:expansion} into the Vlasov equation \eqref{eq:vlasov}, and matching the coefficients on both sides,  the moment system for the Vlasov equation is derived as 
{\small 
\begin{equation}
\label{eq:moment1}
	\begin{split}
		\frac{\partial f_{\alpha}}{\partial t}
		+&\sum_{d=1}^3\left(\frac{\partial u_d}{\partial t}+\sum_{j=1}^3u_j\frac{\partial u_d}{\partial x_j}-E_d-\sum_{l,m=1}^3\epsilon_{dlm}u_lB_m  \right)f_{\alpha-e_d}   \\
		-&\sum_{d,l,m=1}^3 \epsilon_{dlm}(\alpha_l+1)B_mf_{\alpha-e_d+e_l}
		+\frac{1}{2}\left(\frac{\partial \mathcal{T}}{\partial t}+\sum_{j=1}^3u_j\frac{\partial \mathcal{T}}{\partial x_j}   \right)\sum_{d=1}^3f_{\alpha-2e_d}\\
		+&\sum_{d,j=1}^3\left[ \frac{\partial u_d}{\partial x_j}\left(\mathcal{T}f_{\alpha-e_d-e_j}+(\alpha_j+1)f_{\alpha-e_d+e_j} \right) 
		+ \frac{1}{2}\frac{\partial \mathcal{T}}{\partial x_j}\left(\mathcal{T}f_{\alpha-2e_d-e_j}+(\alpha_j+1)f_{\alpha-2e_d+e_j}  \right)	\right]\\
		+&\sum_{j=1}^3\left(\mathcal{T}\frac{\partial f_{\alpha-e_j}}{\partial x_j}+
		u_j\frac{\partial f_{\alpha}}{\partial x_j}
		+(\alpha_j+1)\frac{\partial f_{\alpha+e_j}}{\partial x_j}\right) = 0,\qquad\forall|\alpha|\geqslant  0.
	\end{split}
\end{equation}
}

Following the routine in \cite{Vlasov2012},  we can deduce the equations for the density, macroscopic velocity and temperature from \eqref{eq:moment1} by letting $\alpha = \bz, e_d$ and $2e_d, d = 1, 2,3$, as
\begin{equation}
\label{eq:moment_rhouE}
\begin{gathered}
	\frac{\partial \rho}{\partial t}+
	\sum_{j=1}^3\left(u_j\frac{\partial \rho}{\partial x_j}+\rho\frac{\partial u_j}{\partial x_j}\right)=0, \\
		\frac{\partial u_d}{\partial t}+\sum_{j=1}^3u_j\frac{\partial u_d}{\partial x_j}+
	\frac{1}{\rho}\sum_{j=1}^3\frac{\partial p_{jd}}{\partial x_j}=E_d+\sum_{l,m=1}^3 \epsilon_{dlm}u_lB_m, \\
	\rho\left(
	\frac{\partial \mathcal{T}}{\partial t}
	+\sum_{j=1}^3u_j\frac{\partial \mathcal{T}}{\partial x_j}\right)
	+\frac{2}{3}\sum_{j=1}^3
	\left(\frac{\partial q_j}{\partial x_j}+\sum_{d=1}^3 p_{j d}\frac{\partial u_d}{\partial x_j} \right)=0.
\end{gathered}
\end{equation}
Moreover, substituting \eqref{eq:moment_rhouE} into \eqref{eq:moment1} to eliminate the terms with temporal derivatives of $u_d$ and $\mathcal{T}$, 
it holds for the high order moment coefficients that 
{\small 
\begin{equation}
\label{eq:moment}
	\begin{split}
		\frac{\partial f_{\alpha } }{\partial t}
		+&\sum_{j=1}^3\left(\mathcal{T}\frac{\partial f_{\alpha-e_j}}{\partial x_j}+
		u_j\frac{\partial f_{\alpha}}{\partial x_j}
		+(\alpha_j+1)\frac{\partial f_{\alpha+e_j}}{\partial x_j}\right)\\
		+&\sum_{d,j=1}^3\left[\mathcal{T}f_{\alpha-e_d-e_j}+(\alpha_j+1)f_{\alpha-e_d+e_j} 
		-\frac{p_{jd}}{3\rho}\sum_{l=1}^3 f_{\alpha-2e_l} 
		\right]\frac{\partial u_d }{\partial x_j}\\
		-&\sum_{d,j=1}^3 \frac{f_{\alpha-e_d}}{\rho}\frac{\partial p_{jd}}{\partial x_j}
		-\frac{1}{3\rho}\left(\sum_{k=1}^3f_{\alpha-2e_k}\right)\sum_{j=1}^3 \frac{\partial q_j}{\partial x_j}\\
		+&\sum_{j=1}^3\left[\sum_{d=1}^3\mathcal{T}f_{\alpha-2e_d-e_j}+(\alpha_j+1)f_{\alpha-2e_d+e_j}  \right]	
		\left(-\frac{\mathcal{T}}{2\rho}\frac{\partial \rho}{\partial x_j}+
		\frac{1}{6\rho}\sum_{d=1}^3\frac{\partial p_{dd}}{\partial x_j}\right)\\
		=&\sum_{d,l,m=1}^3 \epsilon_{dlm}(\alpha_l+1)B_m f_{\alpha-e_d+e_l}
		,\quad\forall|\alpha|\geqslant 2.
		\end{split}
\end{equation}
}


 Collecting \eqref{eq:moment_rhouE} and \eqref{eq:moment},  we obtain the moment equations of the VM system with infinite number of equations. In the numerical simulation, a truncation should be adopted for the expansion of the distribution function in \eqref{eq:expansion}, rendering \eqref{eq:moment} a finite moment system.  A regularization method is further applied to obtain a closed moment system, which we will discuss in the next section.

\subsection{Closure of the moment system} 
To obtain a finite system, the expansion \eqref{eq:expansion} is truncated as \begin{equation}
	f(t,\bm{x},\bm{v})\approx \sum_{|\alpha|\leqslant  M}f_{\alpha}(t,\bm{x})
	\mathcal{H}_{\mathcal{T},\alpha}
	\left(\frac{\bm{v}-\bm{u}(t,\bm{x})}{\sqrt{\mathcal{T}(t,\bm{x})}}\right),
\end{equation}
where $M$ is the truncation order, and $|\alpha| = \alpha_1 + \alpha_2 + \alpha_3$. The resulted finite  moment system is further closed by adopting the regularization proposed in \cite{2014Globally}. 
Let $\bbf = (\rho, \bu, \mT, f_{2e_1}, \dots )$. Substituting the terms $f_{\alpha}, |\alpha| = M+1$ with the regularized term in \cite{2014Globally}, the quasi-linear closed moment system can be rewritten as 
\begin{equation}
\label{eq:regularization}
\pd{\bbf}{t} + \sum_{j=1}^3 \hat{\bm{M}}_j(\bbf) \pd{\bbf}{x_j} = \bm{G}\bbf + \bm{g},
\end{equation}
with
\begin{gather}
\label{eq:global}
	\hat{\bm{M}}_j(\bbf)\dfrac{\partial \bbf}{\partial x_j}
	=	\bm{M}_j(\bbf)\dfrac{\partial \bbf}{\partial x_j}
	-\sum_{|\alpha|=M}\mathcal{R}_M^j(\alpha)I_{\mathcal{N}(\alpha)},\\
	\mathcal{R}_M^j(\alpha)=(\alpha_j+1)\left[
	\sum_{d=1}^{D} f_{\alpha-e_d+e_j}\frac{\partial u_d}{\partial x_j}
	+\frac{1}{2}(\sum_{d=1}^{D} f_{\alpha-2e_d+e_j}
	)\frac{\partial \mathcal{T}}{\partial x_j}	\right],\\
		\bm{g}_{\mathcal{N}(e_i)}=E_i,\quad
	\bm{G}_{\mathcal{N}(\alpha),\mathcal{N}(\alpha-e_d+e_l)}
	=\sum_{m=1}^3\epsilon_{d l m}(\alpha_l+1)B_m,\quad d,l=1,2,3, \label{eq:global:3}
\end{gather}
where 
\begin{equation}
		\mathcal{N}(\alpha)
		=\sum_{i=1}^3
		\begin{pmatrix}
		\sum_{k=4-i}^3 \alpha_k+i-1\\
		i
		\end{pmatrix}
		+1.
\end{equation}
Here $\bm{M}_j$ is an $\mathcal{N} \times \mathcal{N}$ matrix with $\mathcal{N}=C^3_{M+3}$, corresponding to the terms with derivatives of $\bbf$, the detailed form of which can be derived from the moment system \eqref{eq:moment_rhouE} and \eqref{eq:moment}. $I_k$  is the $k$-th column of the $\mathcal{N} \times \mathcal{N}$ identity matrix.
$\bm{G}$ is an $\mathcal{N} \times \mathcal{N}$ matrix and $\bm{g}$ is a vector of length $\mathcal{N}$, whose entries are given in \eqref{eq:global:3} while all other entries are zero. 
We refer readers to \cite{2014Globally} for the detailed derivation of this  moment system and the study on  the global hyperbolicity of \eqref{eq:regularization}.
We show here the most important result in the following lemma.
\begin{lemma}	For any unit vector $\bm{n} = (n_1,n_2,n_3)^T \in \bbR^3$, the matrix 
	$\sum^3_{j=1}n_j\hat{\bm{M}}_j$ is diagonalizable.
	Precisely, its characteristic polynomial is
	\begin{equation}
		\left| \lambda\bm{\text{I}}-
		\sum_{j=1}^3 n_j\hat{\bm{M}}_j\right|
		=\prod^{M}_{i=0} \prod^{i}_{j=0}
		\mathcal{T}^{(j+1)/2}\He_{j+1}\left(\frac{\lambda-\bm{u}\cdot\bm{n}}{ \sqrt {\mathcal{T} } } \right),
	\end{equation}
	and its eigenvalues are 
	\begin{equation}
		\bm{u}\cdot\bm{n}+C_{i,j}\sqrt{\mathcal{T}},\quad
		1\leqslant i\leqslant j\leqslant M+1,
	\end{equation}
	where $C_{i,j}$ is the $i$-th root of $\He_j(z)$.
\end{lemma}

Up to  now, we have obtained the closed moment equations for the VM system. It is worth mentioning that different from the VP system studied in \cite{Vlasov2012}, the magnetic field $\bm{B}$ in the VM system is coupled in the governing equations of $f_{\alpha}$. Thus, additional  care needs to  be taken to ensure the conservation properties.
 The energy-conserving numerical scheme will be introduced in the next section to solve the moment system coupling with Maxwell's equations \eqref{eq:maxwell}. 

\section{Energy-conserving numerical schemes}
\label{sec:scheme}
In this section, we propose the energy-conserving numerical schemes for the VM system, with the temporal discrete schemes discussed in Sec. \ref{sec:semi} and fully discrete schemes discussed in Sec. \ref{sec:full}.

\subsection{Temporal discrete schemes}
\label{sec:semi}
Inspired by the numerical method for the VP system in \cite{Vlasov2012}, the  Strang-splitting method is utilized to solve the VM system. Precisely, the Vlasov equation is split into the following two parts:
\begin{itemize}
  \item the convection step
  \begin{equation}
  \label{eq:convection}
	\dfrac{\partial f}{\partial t}+\bm{v}\cdot\nabla_{\bm{x} }f=0,
\end{equation}
\item the Lorentz force step 
\begin{equation}
 \label{eq:Lorentz}
	\dfrac{\partial f}{\partial t}+(\bm{E}+\bm{v}\times\bm{B})\cdot \nabla_{\bm{v}} f=0.
\end{equation}
\end{itemize}

In the framework of the regularized moment method discussed in
Sec. \ref{sec:method}, the governing equation \eqref{eq:Lorentz} can be reduced into several ODEs based on \eqref{eq:moment_rhouE} and \eqref{eq:moment}, yielding
\begin{align}
  \label{eq:Lorentz_u} 
  \dfrac{\dd \bm{u}}{\dd t} &=\bm{E}+\bm{u}\times\bm{B}, \\
  \label{eq:Lorentz_E} 
  \frac{\dd \mathcal{T}}{\dd t}&=0,\\
  \label{eq:Lorentz_f}
  	\dfrac{\dd f_{\alpha}}{\dd t} &=\sum_{d,l,m=1}^3
	\epsilon_{dlm}(\alpha_l+1)B_m f_{\alpha-e_d+e_l},\qquad 2\leqslant |\alpha|\leqslant M.
\end{align}

It is obvious the movement of particles subjected to the Lorentz force in the moment system is very concise. Compared to the general Hermite spectral method, it is surprised to find that \eqref{eq:Lorentz_u} is Newton's first law of motion in classical mechanics. We also see that the high-order moment coefficients $f_{\alpha}$ are only related to the magnetic field and the moment coefficients in the same order of $\alpha$. Although, there are few physical definitions of the high-order moment coefficients, we still expect more physical explanations for \eqref{eq:Lorentz_f} in the future.


We now introduce two energy-conserving temporal schemes for \eqref{eq:convection} and \eqref{eq:Lorentz_u}-\eqref{eq:Lorentz_f} coupling with Maxwell's equations \eqref{eq:maxwell}. As for the convection step \eqref{eq:convection}, any implicit or explicit Runge-Kutta method can be applied to solve it and conserves the kinetic energy \cite{VA_Zhong, Vlasov2012}. We consider the forward Euler method for this step. The Lorentz step \eqref{eq:Lorentz_u}-\eqref{eq:Lorentz_f} and Maxwell's equations \eqref{eq:maxwell} contain the main coupling of the macroscopic velocity and the electromagnetic field. Thus, the key point is how to advance the macroscopic velocity and the electromagnetic field in this step to balance the kinetic and electromagnetic energy. 

For both schemes designed in the following, the electric field and the macroscopic velocity are advanced implicitly, while the magnetic field is advanced  implicitly for one scheme and explicitly for the other. 
The first scheme, denoted by {\bf Scheme-I}, is designed as follows:
\paragraph{\underline{Scheme-I: Implicit for the magnetic field}}
\begin{subequations}
  \label{eq:scheme_1}
	\begin{align}
	  \label{eq:scheme_1_con}
		&\dfrac{f^{n+1,*}-f^{n} }{\Delta t}=-\bm{v}\cdot\nabla_{\bm{x}} f^{n},\\
		\label{eq:scheme_1_B}
		&\dfrac{\bm{B}^{n+1}-\bm{B}^n}{\Delta t }=-\nabla_{\bm{x}} \times \dfrac{\bm{E}^n+\bm{E}^{n+1}}{2},\\
		\label{eq:scheme_1_E}
		&\dfrac{\bm{E}^{n+1}-\bm{E}^n}{\Delta t}=\nabla_{\bm{x}} \times \dfrac{\bm{B}^{n}+\bm{B}^{n+1}}{2}-\bm{J}^{n+1/2},\\
		\label{eq:scheme_1_u}
		&\frac{\bm{u}^{n+1}-\bm{u}^{n+1,*}}{\Delta t}=\frac{\bm{E}^n+\bm{E}^{n+1}}{2}
			+\frac{\bm{u}^{n+1,*}+\bm{u}^{n+1}}{2}\times \frac{\bm{B}^{n}+\bm{B}^{n+1}}{2},\\
			\label{eq:scheme_1_f}
			&		\frac{f_{\alpha}^{n+1}-f_{\alpha}^{n+1,*}}{\Delta t}\,=
		\sum_{d,l,m=1}^{3} \epsilon_{d l m}	(\alpha_l+1)\frac{B_m^{n}+B_m^{n+1}}{2}
	 f^{n+1,*}_{\alpha-e_d+e_l}, \qquad 2\leqslant \vert \alpha \vert\leqslant M.
	\end{align}
\end{subequations}
Here $\bm{J}^{n+1/2}$ is defined as 
\begin{equation}
\label{eq:scheme1_J}
   \bm{J}^{n+1/2}=\int_{\bbR^3} \bv
		\left( \frac{f^{n+1, \ast}+f^{n+1} }{2} \right)\dd \bv = \frac{\rho^{n+1} \bu^{n+1} + \rho^{n+1, \ast}\bu^{n+1,\ast}}{2},
\end{equation}
where $\rho^{n+1,\ast}$ and $\bu^{n+1,\ast}$ are the density and the macroscopic velocity connected with $f^{n+1,*}$ in \eqref{eq:scheme_1_con} via \eqref{eq:macro}.
 The density $\rho$ and the thermal temperature $\mT$ remain unchanged during the Lorentz step, i.e., 
\begin{equation}
  \label{eq:rho_T_n+1}
  \rho^{n+1} = \rho^{n+1, \ast},
  \qquad
  \mT^{n+1} = \mT^{n+1, \ast},
\end{equation}
which renders \eqref{eq:scheme1_J} as
\begin{equation}
\label{eq:scheme1_J_final}
   \bm{J}^{n+1/2}= \frac{\rho^{n+1, \ast}(\bu^{n+1} + \bu^{n+1,\ast})}{2}. 
\end{equation}

\begin{theorem}[Total energy conservation of {\bf Scheme-I} ]
\label{thm:energy_1}
  {\bf Scheme-I} preserves the discrete total energy 
  \begin{equation}
  \label{eq:con_E}
 	\mEt^{n+1}= \mEt^{n},
 	\end{equation}
 for the VM system in Sec. \ref{sec:FPL} with periodic boundary conditions in $\bm{x}$, where
	\begin{equation}
  \label{eq:semi_dis_energy_1} 
 \mEt^n = \frac{1}{2} \int_{\Omega} \rho^n (|\bu^n|^2 + 3 \mT^n) \dd \bx + \frac{1}{2} \int_{\Omega} |\bE^n|^2 + |\bB^n|^2 \dd \bx. 
\end{equation}
\end{theorem}

\begin{proof}
Define the kinetic energy $\mE^n_K$, electric energy $\mE_E^n$ and magnetic energy $\mE_B^n$ at time $t^n$ as 
\begin{equation}
  \label{eq:semi_KEB}
  \mE^n_K = \frac{1}{2} \int_{\Omega} \rho^n (|\bu^n|^2 + 3 \mT^n) \dd \bx, \qquad \mE_E^n = \frac{1}{2} \int_{\Omega} |\bE^n|^2 \dd \bx, \qquad \mE_B^n = \frac{1}{2} \int_{\Omega} |\bB^n|^2 \dd \bx.
\end{equation}
It follows from \eqref{eq:macro} that the kinetic energy $\mE^n_K$ can also be written as 
\begin{equation}
  \label{eq:semi_K}
  \mE^n_K = \frac{1}{2} \int_{\Omega} \int_{\bbR^3}|\bv|^2 f^n \dd \bv \dd \bx. 
\end{equation}

 For the convection step, multiplying \eqref{eq:scheme_1_con} with $|\bv|^2$ and
 integrating with respect to $\bv$ and $\bx$ over $\Omega\times\mathbb{R}^3$, yielding
\begin{align}
\label{eq:prove_con_1}
\begin{split}
 \int_{\bbR^3}\int_{\Omega} \dfrac{f^{n+1,\ast} - f^n}{\Delta t} |\bv|^2 \dd \bx \dd \bv 
&= - \int_{\bbR^3}\int_{\Omega} \bv \cdot \nabla_{\bx} f^n |\bv|^2 \dd \bx \dd \bv \\&= -\int_{\bbR^3} |\bv|^2 \bv \cdot \int_{\Omega} \nabla_{\bx} f^n\dd \bx \dd \bv=0,
\end{split}
\end{align}
where the last equality is due to the periodic boundary conditions in $\bm x$. Therefore,
\begin{equation}
  \label{eq:prove_con_3}
  \mE_K^{n+1, \ast} = \mE_K^{n}. 
\end{equation}

For the Lorentz force step, the thermal energy is unchanged according to \eqref{eq:rho_T_n+1}, i.e.,
\begin{equation}
  \frac{3}{2} \int_{\Omega} \rho^{n+1} \mT^{n+1} \dd \bx=
  \frac{3}{2} \int_{\Omega} \rho^{n+1,\ast} \mT^{n+1,\ast} \dd \bx.
  \label{eq:prove_lor_T}
\end{equation}
On the other hand, multiplying \eqref{eq:scheme_1_B} with $\bB^{n+1} + \bB^n$, \eqref{eq:scheme_1_E} with $\bE^{n+1} + \bE^{n}$, \eqref{eq:scheme_1_u} with $\rho^{n+1}(\bu^{n+1} + \bu^{n+1, \ast})$, and integrating with respect to $\bx$ over $\Omega$, we obtain 
{\small 
\begin{equation}
  \label{eq:prove_lor_1}
  \begin{aligned}
    & \int_{\Omega} \dfrac{|\bB^{n+1}|^2 - |\bB^{n}|^2 }{\Delta t} \dd \bx = - \int_{\Omega} \nabla_{\bx} \times \frac{\bE^{n+1} + \bE^n}{2} \cdot (\bB^{n+1} + \bB^{n})\dd \bx, \\
    & \int_{\Omega} \dfrac{|\bE^{n+1}|^2 - |\bE^{n}|^2 }{\Delta t} \dd \bx = \int_{\Omega} \nabla_{\bx} \times \frac{\bB^{n+1} + \bB^n}{2} \cdot (\bE^{n+1} + \bE^{n})\dd \bx - \int_{\Omega}\bJ^{n+1/2} \cdot (\bE^{n+1} + \bE^{n})\dd \bx, \\
    & \int_{\Omega} \dfrac{\rho^{n+1}(|\bu^{n+1}|^2 - |\bu^{n+1,\ast}|^2)}{\Delta t} \dd \bx = \int_{\Omega}\rho^{n+1} \frac{\bE^{n+1} + \bE^n}{2} \cdot (\bu^{n+1} + \bu^{n+1,\ast})\dd \bx \\
     & \hspace{5cm} + 
     \int_{\Omega}\rho^{n+1}\frac{(\bu^{n+1} + \bu^{n+1,\ast})}{2} \times \frac{\bB^{n+1} + \bB^n}{2} \cdot (\bu^{n+1} + \bu^{n+1,\ast})\dd \bx.
  \end{aligned}
\end{equation}
}
Summing up \eqref{eq:prove_lor_T} and \eqref{eq:prove_lor_1}, together with \eqref{eq:rho_T_n+1} and \eqref{eq:scheme1_J_final}, we have
\begin{equation}
  \label{eq:prove_lor_2}
\begin{aligned}
  &(\mE^{n+1}_K + \mE^{n+1}_E + \mE^{n+1}_B) - (\mE^{n+1, \ast}_K + \mE^{n}_E + \mE^n_B)
  \\ & \qquad =\frac{\Delta t}{2} \int_{\Omega} \nabla_{\bx} \times (\bB^{n+1} + \bB^n) \cdot (\bE^{n+1} + \bE^{n})
  -\nabla_{\bx} \times(\bE^{n+1} + \bE^{n})\cdot
  (\bB^{n+1} + \bB^n)
  \dd \bx \\
  & \qquad = \frac{\Delta t}{2}\int_{\Omega} \nabla_{\bx} \cdot [(\bB^{n+1} + \bB^n) \times (\bE^{n+1} + \bE^{n})] \dd \bx=0,
  \end{aligned}
\end{equation}
 where the last equality holds due to the periodic boundary conditions in $\bm x$.
We complete the proof by combining \eqref{eq:prove_lor_2} with \eqref{eq:prove_con_3}, yielding
\begin{equation}
  \label{eq:prove_lor_4}
  \mEt^{n+1} \triangleq (\mE^{n+1}_K + \mE^{n+1}_E + \mE^{n+1}_B) = (\mE^{n+1, \ast}_K + \mE^{n}_E + \mE^n_B) \triangleq \mEt^{n}.
\end{equation}
\end{proof}

From this theorem, we can see that {\bf Scheme-I} exactly preserves the total energy. This scheme can potentially work for the VM system when stiffness occurs in the electromagnetic field, since \eqref{eq:scheme_1_B}, \eqref{eq:scheme_1_E} and \eqref{eq:scheme_1_u} are formulated using the implicit midpoint method on $(\bm E,\bm B,\bm u)$. However, the computation of this scheme is demanding and requires inversion of a nonlinear high-dimensional coupled system. To improve the efficiency and reduce computational cost, we modify {\bf Scheme-I} by advancing the magnetic field explicitly,  and denote the resulting scheme as {\bf Scheme-II}, which is designed as follows:

\paragraph{\underline{Scheme-II: Explicit for the magnetic field}}
\begin{subequations}
\label{eq:scheme_2}
	\begin{align}
	\label{eq:scheme_2_con}
		&\frac{f^{n+1,*}-f^{n} }{\Delta t}=-\bm{v}\cdot\nabla_{\bm{x}}f^n,\\
		\label{eq:scheme_2_B1}
		&\frac{\bm{B}^{n+1/2}-\bm{B}^n}{\Delta t /2}=-\nabla_{\bm{x}} \times \bm{E}^n,\\	
		\label{eq:scheme_2_E}
		&\frac{\bm{E}^{n+1}-\bm{E}^n}{\Delta t}= \nabla_{\bm{x}}\times\bm{B}^{n+1/2}-\bm{J}^{n+1/2},\\
		 \label{eq:scheme_2_B2}
		&\frac{\bm{B}^{n+1}-\bm{B}^{n+1/2}}{\Delta t /2}=-\nabla_{\bm{x}} \times \bm{E}^{n+1}, \\
		\label{eq:scheme_2_u}
		&\frac{\bm{u}^{n+1}-\bm{u}^{n+1,*}}{\Delta t}\ \ =\frac{\bm{E}^{n}+\bm{E}^{n+1}}{2}
		+\frac{\bm{u}^{n+1,*}+\bm{u}^{n+1}}{2}\times \bm{B}^{n+1/2},\\
		\label{eq:scheme_2_f}
		&
		\frac{f_{\alpha}^{n+1}-f_{\alpha}^{n+1,*}}{\Delta t}=
		\sum_{d,l,m=1}^{3} \epsilon_{d l m}(\alpha_l+1)B_m^{n+1/2}
		 f^{n+1,*}_{\alpha-e_d+e_l}, \qquad 2\leqslant \vert \alpha \vert\leqslant M,
	\end{align}
\end{subequations}
where $\bJ^{n+1/2}$ is defined the same as in \eqref{eq:scheme1_J}.  $\rho^{n+1,\ast},\ \bu^{n+1,\ast}$ are connected with $f^{n+1,*}$ in \eqref{eq:scheme_2_con} and \eqref{eq:macro}. \eqref{eq:rho_T_n+1} and \eqref{eq:scheme1_J_final} also hold here. 

\begin{theorem}[Total energy conservation of {\bf Scheme-II}] 
\label{thm:energy_2}
  {\bf Scheme-II} preserves the discrete total energy
	$\mEt^n=\mEt^{n+1}$ of the VM system in Sec. \ref{sec:FPL} with periodic boundary conditions in $\bm x$, where
\begin{equation}
   \label{eq:semi_dis_energy_2} 
  \mEt^n = \frac{1}{2} \int_{\Omega} \rho^n (|\bu^n|^2 + 3 \mT^n) \dd \bx + \frac{1}{2} \int_{\Omega} |\bE^n|^2 + \bB^{n+1/2} \cdot \bB^{n-1/2}\dd \bx. 
 \end{equation}
\end{theorem}

\begin{proof}
The change of the kinetic energy in the convection step and the 
thermal energy in the Lorentz force step are the same as \eqref{eq:prove_con_3} and \eqref{eq:prove_lor_T} in Theorem \ref{thm:energy_1}. 
For the Lorentz force step, at time level $n+1$, \eqref{eq:scheme_2_B1} becomes
\begin{equation}
  \label{eq:prove_thm2_1}
  \frac{\bm{B}^{n+3/2}-\bm{B}^{n+1}}{\Delta t /2}=-\nabla_{\bm{x}} \times \bm{E}^{n+1},
\end{equation}
which, combining with \eqref{eq:scheme_2_B2}, yields
\begin{equation}
  \label{eq:prove_thm2_2}
  \frac{\bm{B}^{n+3/2}-\bm{B}^{n+1/2}}{\Delta t}=-\nabla_{\bm{x}} \times \bm{E}^{n+1}, \qquad \frac{\bm{B}^{n+1/2}-\bm{B}^{n-1/2}}{\Delta t}=-\nabla_{\bm{x}} \times \bm{E}^{n}. 
\end{equation}
Multiplying \eqref{eq:prove_thm2_2} with $\bB^{n+1/2}$, \eqref{eq:scheme_2_E} with $\bE^{n+1} + \bE^{n}$, \eqref{eq:scheme_2_u} with $\rho^{n+1}(\bu^{n+1} + \bu^{n+1, \ast})$, and integrating with respect to $\bx$ over $\Omega$, we obtain
{\small 
\begin{equation}
  \label{eq:prove_thm2_3}
  \begin{aligned}
    & \int_{\Omega} \dfrac{(\bB^{n+3/2} - \bB^{n+1/2}) \cdot \bB^{n+1/2}}{\Delta t} \dd \bx = - \int_{\Omega} \nabla_{\bx} \times \bE^{n+1} \cdot \bB^{n+1/2}\dd \bx, \\
    & \int_{\Omega} \dfrac{(\bB^{n+1/2} - \bB^{n-1/2}) \cdot \bB^{n+1/2}}{\Delta t} \dd \bx = - \int_{\Omega} \nabla_{\bx} \times \bE^{n} \cdot \bB^{n+1/2}\dd \bx, \\
    & \int_{\Omega} \dfrac{|\bE^{n+1}|^2 - |\bE^{n}|^2 }{\Delta t} \dd \bx = \int_{\Omega} \nabla_{\bx} \times \bB^{n+1/2} \cdot (\bE^{n+1} + \bE^{n})\dd \bx - \int_{\Omega}\bJ^{n+1/2} \cdot (\bE^{n+1} + \bE^{n})\dd \bx, \\
    & \int_{\Omega} \dfrac{\rho^{n+1}(|\bu^{n+1}|^2 - |\bu^{n+1,\ast}|^2)}{\Delta t} \dd \bx = \int_{\Omega}\rho^{n+1} \frac{\bE^{n+1} + \bE^n}{2} \cdot (\bu^{n+1} + \bu^{n+1,\ast})\dd \bx \\
     & \hspace{5cm} + 
     \int_{\Omega}\rho^{n+1}\frac{(\bu^{n+1} + \bu^{n+1,\ast})}{2} \times \bB^{n+1/2} \cdot (\bu^{n+1} + \bu^{n+1,\ast})\dd \bx.
  \end{aligned}
\end{equation}
}
Define the magnetic engery $ \mE^{n+1}_B$ as 
\begin{equation}
  \label{eq:prove_thm2_5}
   \mE^{n}_B = \frac{1}{2} \int_{\Omega} \bB^{n+1/2} \cdot \bB^{n-1/2}\dd \bx.
\end{equation}
 Summing up \eqref{eq:prove_thm2_3} and \eqref{eq:prove_lor_T}, together with \eqref{eq:rho_T_n+1} and \eqref{eq:scheme1_J_final}, we have
\begin{equation}
  \label{eq:prove_thm2_4}
\begin{aligned}
  &(\mE^{n+1}_K + \mE^{n+1}_E + \mE^{n+1}_B) - (\mE^{n+1, \ast}_K + \mE^{n}_E + \mE^n_B)\\ 
  & \qquad = \Delta t \int_{\Omega} \nabla_{\bx} \times \bB^{n+1/2} \cdot (\bE^{n+1} + \bE^{n})
  -\nabla_{\bx} \times(\bE^{n+1} + \bE^{n})\cdot
  \bB^{n+1/2}
  \dd \bx \\
  & \qquad =\Delta t\int_{\Omega} \nabla_{\bx} \cdot[ \bB^{n+1/2} \cdot (\bE^{n+1} + \bE^{n})] \dd \bx = 0,
  \end{aligned}
\end{equation}
where the last equality is due to periodic boundary conditions in $\bm x$.
We complete the proof. 
\end{proof}

Similar as studied in \cite{VM_Zhong}, {\bf Scheme-II} achieves near conservation of the total energy. The total energy \eqref{eq:semi_dis_energy_2} of 
{\bf Scheme-II} is a modified  version of the exact total energy defined in
\eqref{eq:semi_dis_energy_1}. This ensures that over the long run, the numerical energy will not deviate much from its actual value. Total energy conservation is preserved with a suitable time step size in the numerical simulations.

It is worth mentioning that although {\bf Scheme-II} is formulated by using the leap frog method for Maxwell's equations in \eqref{eq:scheme_2_B1}-\eqref{eq:scheme_2_B2},  $\bE^{n+1}$ and $\bu^{n+1}$ are advanced implicitly via \eqref{eq:scheme_1_E} and \eqref{eq:scheme_2_u}, since $\bm J^{n+1/2}$ defined in \eqref{eq:scheme1_J} involves the information of $\bu^{n+1}$.
Comparing with {\bf Scheme-I}, {\bf Scheme-II} deals with a smaller linear system, where the matrix form for updating $\bu$ and $\bE$ to the next time level is given by
{\small
\begin{equation}
\renewcommand\arraystretch{1.3}
\label{eq:scheme_matrix}
		\left(I+\frac{\Delta t}{2}
		\begin{bmatrix}
		\begin{array}{cccc|ccc}
		 	& & & & \rho^{n+1}& & \\
			& &\text{\huge 0}& & &\rho^{n+1} & \\
			& & & & &  &\rho^{n+1} \\[1mm] \hline
			&-1 &  & & 0  &-B_3^{n+1/2} &B_2^{n+1/2}\\
			&  &-1 & &  B_3^{n+1/2} &0  &-B_1^{n+1/2}\\
			&  &  &-1&  -B_2^{n+1/2}&B_1^{n+1/2} &0
		\end{array}
	 \end{bmatrix} 
		\right)
		\begin{bmatrix}
			E_1^{n+1}\\
			E_2^{n+1}\\
			E_3^{n+1}\\
			u_1^{n+1}\\
			u_2^{n+1}\\
			u_3^{n+1}
		\end{bmatrix}
		= {\rm RHS}^{n+1,\ast}.
\end{equation}
}

\begin{remark}
By setting $\bB=\bm 0$ in {\bf Scheme-I} and {\bf Scheme-II}, both energy-preserving schemes can be applied to the VA system \eqref{eq:vp} in the framework of moment methods, given as follows
\begin{subequations}
\label{eq:scheme_3}
	\begin{align}
	\label{eq:scheme_3_con}
		&\frac{f^{n+1,*}-f^{n} }{\Delta t}=-\bm{v}\cdot\nabla_{\bm{x}}f^n,\\
		\label{eq:scheme_3_E}
		&\frac{\bm{E}^{n+1}-\bm{E}^n}{\Delta t}= -\bm{J}^{n+1/2},\\
		\label{eq:scheme_3_u}
		&\frac{\bm{u}^{n+1}-\bm{u}^{n+1,*}}{\Delta t}\ \ =\frac{\bm{E}^{n}+\bm{E}^{n+1}}{2}.
	\end{align}
\end{subequations}
\end{remark}

\subsection{Fully discrete schemes and their properties}
\label{sec:full}
In this section, we formulate the fully discrete schemes and discuss their conservation properties.
\subsubsection{Fully discrete schemes}
In this section, we describe the details of the spatial discretization coupling with the temporal discretization   {\bf Scheme-II} to formulate the fully discrete scheme. All the discussions can be applied to {\bf Scheme-I} as well.

First, the spatial domain $\Omega\subset\mathbb{R}^3$ is taken a uniform partition into cubic meshes $\{T_{\bj}\}$ with $\bj = (j_1, j_2, j_3)$. Denote the mesh size as $\Delta \bx = (\Delta x_1, \Delta x_2, \Delta x_3)$. Let $f_{\bj}^n(\bv)$ denote the numerical approximation to $f(t,\bx, \bv)$ in the mesh $T_{\bj}$ at time $t^n$. 
The Hermite expansion for $f_{\bj}^n(\bv)$ is 
\begin{equation}
  f^{n}_{\bm{j}}(t,\bm{v})
  =\sum_{|\alpha|\leqslant M}
  f_{\alpha, \bj}^n
  \mathcal{H}_{\mathcal{T}_{\bm{j}}^n,\alpha}
  \left(  \frac{\bm{v}-\bm{u}^n_{\bm{j}}
  }{\sqrt{\mathcal{T}_{\bm{j}}^n}} \right).
\end{equation}
We further denote $\bm{E}_{\bm{j}}^{n}$ and $\bB_{\bj}^n$ as the numerical approximations to the electric field $\bE$ and the magnetic field $\bB$ in the mesh $T_{\bm{j}}$ at time $t^n$, respectively.

For the convection step, we follow the numerical scheme in \cite{Vlasov2012, wang2017VPFP}  with the fully discrete scheme for $f^n_{\bj}$ given by 
\begin{equation}
\label{eq:update_f}
\frac{f^{n+1,*}_{\bm{j}}-f^{n}_{\bm{j}}}{\Delta t}=-\sum_{d=1}^{3}\frac{1}{\Delta x_d}
\left(F^n_{\bm{j}+e_d/2}- F^n_{\bm{j}-e_d/2}\right)
-\sum_{d=1}^{3}\frac{1}{\Delta x_d}
\left(R^{n-}_{\bm{j}+e_d/2}- R^{n+}_{\bm{j}-e_d/2}\right),
\end{equation}
with $F^n_{\bj \pm e_d/2}$ and $R^{n\pm}_{\bm{j}+e_d/2}$ being the numerical fluxes.
For the non-conserved flux $R^{n\pm}_{\bm{j}+e_d/2}$, we adopt the same recipe in \cite{wang2017VPFP} and thus omit the details here. For the flux $F^n_{\bj \pm e_d/2}$, we adopt the HLL flux 
{\small 
\begin{equation}
\label{eq:hll_flux}
	F^n_{\bm{j}+e_d/2}=
	\left\{
	\begin{array}{ll}
	 	v_d f_{\bm{j}+e_d/2}^{n,(L)}, &\lambda^L_{\bm{j}+e_d/2}\geqslant 0,\\[2mm]
		\frac{
		\lambda^R_{\bm{j}+e_d/2}v_d f_{\bj+e_d/2}^{n,(L)} - 
			\lambda^L_{\bm{j}+e_d/2}v_d f_{\bj+e_d/2}^{n,(R)} + 
				\lambda^L_{\bm{j}+e_d/2}\lambda^R_{\bm{j}+e_d/2}( f_{\bj+e_d/2}^{n,(R)} - f_{\bj+e_d/2}^{n,(L)})
		}{\lambda^R_{\bm{j}+e_d/2} -\lambda^L_{\bm{j}+e_d/2}}, 
		& \lambda^L_{\bm{j}+e_d/2}<0<\lambda^R_{\bm{j}+e_d/2},\\[2mm]
		v_d f_{\bm{j}+e_d/2}^{n,(R)}, &\lambda^R_{\bm{j}+e_d/2}\leqslant 0,
	\end{array}
	\right.
\end{equation}
}
where $f_{\bm{j}+e_d/2}^{n,(L)}$ and $f_{\bm{j}+e_d/2}^{n,(R)}$ are the linear reconstruction of $f_{\bj}^n$
\begin{equation}
  \label{eq:recon}
  f_{\bm{j}+e_d/2}^{n,(L)}=
  f^n_{\bm{j}}+\frac{\Delta x_d}{2}\frac{f^n_{\bm{j}+e_d}- f^n_{\bm{j}-e_d}}{2\Delta x_d},\qquad 
  f_{\bm{j}-e_d/2}^{n,(R)}=
  f^n_{\bm{j}}-\frac{\Delta x_d}{2}\frac{f^n_{\bm{j}+e_d}- f^n_{\bm{j}-e_d}}{2\Delta x_d}.
\end{equation}
Here $\lambda^L_{\bm{j}+e_d/2}$, $\lambda^R_{\bm{j}+e_d/2}$ are the minimum and maximum characteristic velocities of the moment system, given by
\begin{equation}
\label{eq:lambda}
\begin{aligned} 
\lambda^L_{\bm{j}+e_d/2} &=\min\left\{ u_{d, \bm{j}}-\text{C}_{M+1}\sqrt{\mathcal{T}^n_{\bm{j}}},
			           u_{d, \bm{j}+e_d}-\text{C}_{M+1}\sqrt{\mathcal{T}^n_{\bm{j}+e_d}}
			 \right\}, \\
\lambda^R_{\bm{j}+e_d/2}&=\max\left\{ u_{d, \bm{j}}+\text{C}_{M+1}\sqrt{\mathcal{T}^n_{\bm{j}}},
			u_{d, \bm{j}+e_d}+\text{C}_{M+1}\sqrt{\mathcal{T}^n_{\bm{j}+e_d}} 
			 \right\},
\end{aligned}
\end{equation}	
 where $\text{C}_{M+1}$ is the largest root of the Hermite polynomial $\He_{M+1}(x)$. 

For the Lorentz force step and Maxwell’s equations, we apply the central finite difference scheme to discretize the spatial variable in 
\eqref{eq:scheme_2_B1}-\eqref{eq:scheme_2_f} of {\bf Scheme-II}, yielding
\begin{subequations}
  \label{eq:full_Lorntz}
  \begin{align}
  \label{eq:full_B1}
  & \frac{ \bm{B}^{n+1/2}_{\bm{j}} - \bm{B}^{n}_{\bm{j}} }
  {\Delta t / 2}
  =- \Pi_{\bx} \bE_{\bj}^{n}, \\
  \label{eq:full_E}
  & \frac{ \bm{E}^{n+1}_{\bm{j}} - \bm{E}^{n}_{\bm{j}} }
  {\Delta t}
  = \Pi_{\bx} \bB_{\bj}^{n+1/2} - \bJ_{\bm{j}}^{n+1/2}, \\
  \label{eq:full_B2}
  &\frac{ \bm{B}^{n+1}_{\bm{j}} - \bm{B}^{n+1/2}_{\bm{j}} }
  {\Delta t / 2}
  =- \Pi_{\bx} \bE_{\bj}^{n+1}, \\
  \label{eq:full_u}
  &\frac{\bu_{\bj}^{n+1} - \bu_{\bj}^{n+1,\ast}}{\Delta t} = \frac{\bE^n_{\bj} + \bE^{n+1}_{\bj}}{2} + \frac{\bu_{\bj}^{n+1, \ast} + \bu_{\bj}^{n+1}}{2} \times \bB^{n+1/2}_{\bj}, \\
  \label{eq:update_f_high}
  & \frac{f_{\alpha, \bj}^{n+1} - f_{\alpha, \bj}^{n+1, \ast}}{\Delta t} = 	\sum_{d,l,m=1}^{3} \epsilon_{d l m}(\alpha_l+1)B_{m,\bm{j}}^{n+1/2}
		 f^{n+1,\ast}_{\alpha-e_d+e_l,\bm{j}}, \qquad 2\leqslant \vert \alpha \vert\leqslant M, 
  \end{align}
\end{subequations}
where $\Pi_{\bx} \cdot$ is the discretization of the curl operator $\nabla_{\bx} \times \cdot $, defined as
\begin{equation}
\label{eq:operator_Pi}
  \Pi_{\bx}\bm{S}_{\bj} = \left|\begin{array}{cccc} 
  \vec{j}_1 &  \vec{j}_2  & \vec{j}_3 \\ 
   {\mathrm D}_1  & {\mathrm D}_2 & {\mathrm D}_3 \\ 
  S_{1,\bm{j}} & S_{2, \bm{j}} & S_{3, \bm{j}}
\end{array}\right|, \qquad {\mathrm D}_d S_{\bj} = \frac{S_{\bj+e_d} - S_{\bj-e_d}}{2\Delta x_d}, \qquad d = 1,2,3.
\end{equation}

The time step $\Delta t$ of the above fully discrete scheme should satisfy the following CFL condition
\begin{equation}
\label{eq:CFL}
	\text{CFL}=\frac{\Delta t^n}{\Delta x_d}\max_{\bm{j}} \left\{ \left|\lambda^R_{\bm{j}+e_d/2}\right|,\left|\lambda^L_{\bm{j}+e_d/2}\right|\right\}<1,\quad d=1,2,3,
\end{equation} 
with $\lambda^L_{\bm{j}+e_d/2}$, $\lambda^R_{\bm{j}+e_d/2}$ given in \eqref{eq:lambda}.
\subsubsection{Outline of the algorithm}
The overall numerical scheme is summarized as follows:
\begin{enumerate}
  \item  Let $n = 0$ and set the initial value of $f_{\alpha,\bj}^0$, $\bE_{\bj}^0$, and $\bB_{\bj}^0$;
  \item Set $\Delta t^n$ according to the CFL condition \eqref{eq:CFL};
  \item Update the convection term to obtain $f_{\bj}^{n+1,\ast}$ using \eqref{eq:update_f};
  \item Obtain $\rho_{\bj}^{n+1, \ast}$, $\mT_{\bj}^{n+1,\ast}$ and $\bu_{\bj}^{n+1,\ast}$ using \eqref{eq:rho} and \eqref{eq:macro};
  \item Update $\bu_{\bj}^{n+1}$, $\bE_{\bj}^{n+1}$, $\bB_{\bj}^{n+1}$ and
  $f_{\alpha,\bj}^{n+1}$ using \eqref{eq:full_Lorntz};
  \item Reset the expansion center and scaling factor in the mesh $T_{\bj}$ with $\bu_{\bj}^{n+1}$ and $\mT_{\bj}^{n+1}$;
  \item Project $f_{\bj}^{n+1}$ to the functional space with expansion center $\bu_{\bj}^{n+1}$ and $\mT_{\bj}^{n+1}$;
  \item Let $n\leftarrow n+1$, and return to Step 2. 
\end{enumerate}

In the framework of the regularized moment method, it is restricted that the expansion center should be the local macroscopic velocity and the scaling factor should be the local temperature. Therefore, in Step 7 of the algorithm, the distribution function is projected into the corresponding space. The total computational cost of this projection is $\mathcal{O}(M^3)$, see e.g. \cite{Vlasov2012} for more details. 

\subsubsection{Conservation properties of fully discrete schemes}
In this subsection, 
 we present conservation properties of fully discrete schemes and defer a rigorous proof of these properties.


\begin{theorem}[Mass conservation]
\label{thm:mass_full}
The fully discrete scheme \eqref{eq:update_f} and \eqref{eq:full_Lorntz}  preserves the discrete mass  of the VM system with periodic boundary conditions in $\bx$, i.e.
\begin{equation}
   \label{eq:proof_mass_1}
   \mP_h^{n+1} = \mP_h^{n},\quad\text{with}\quad \mP_h^n = \sum_{\bj \in \bJ} \Delta V \rho_{\bj}^n,
  \end{equation}
  where $\Delta V = \prod_{d=1}^3 \Delta x_d$ is the mesh volume and $\bJ$ is the total set of mesh indexes. This also holds for the full discrete scheme with time integrators {\bf Scheme-I}.
\end{theorem}

\begin{proof}
The proof is similar to  the mass conservation in \cite{Vlasov2012}, and is thus omitted.
\end{proof}

\begin{theorem}[Total energy conservation]
\label{thm:energy_3}
The fully discrete scheme \eqref{eq:update_f} and \eqref{eq:full_Lorntz}  preserves the discrete total energy $\mE_{{\rm total}, h}^{n+1} =
    \mE_{{\rm total}, h}^n$,
    for the VM system with periodic boundary conditions in $\bx$, where
     \begin{align}
     \label{eq:energy_2}
 \mathcal{E}_{{\rm total}, h}^n =
 			\frac{\Delta V}{2}\sum_{\bm{j}\in\bm{J}} \left(\rho_{\bm{j}}^{n}(\bm{u}_{\bm{j}}^{n})^2+3\rho_{\bm{j}}^{n}\mathcal{T}_{\bm{j}}^{n}
 			+(\bm{E}_{\bm{j}}^{n} )^2+\bm{B}_{\bm{j}}^{n-1/2} \cdot \bm{B}_{\bm{j}}^{n+1/2}\right),
 \end{align}
 and for the fully discrete scheme with  time integrator {\bf Scheme-I}, the numerical energy defined as
     \begin{align}
     \label{eq:energy_1}
 \mathcal{E}_{{\rm total}, h}^n =
 			\frac{\Delta V}{2}\sum_{\bm{j}\in\bm{J}} \left(\rho_{\bm{j}}^{n}(\bm{u}_{\bm{j}}^{n})^2+3\rho_{\bm{j}}^{n}\mathcal{T}_{\bm{j}}^{n}
 			+(\bm{E}_{\bm{j}}^{n} )^2+(\bm{B}_{\bm{j}}^{n})^2\right),
 \end{align}
 is also preserved.
\end{theorem}

\begin{proof}
We only show the proof for the fully discrete scheme with time integrator {\bf Scheme-II}.
The proof for  {\bf Scheme-I} is similar.
We now prove the total energy conservation of the fully discrete scheme  in two steps. 
For the convection step, by multiplying \eqref{eq:update_f} with $|\bv|^2$ and integrating with respect to $\bv$, it holds 
{\small 
\begin{equation}
  \label{eq:proof_energy_2}
   \int_{\bbR^3} |\bv|^2\frac{ f_{\bj}^{n+1,\ast} - f_{\bj}^n}{\Delta t} \dd \bv =\int_{\bbR^3} |\bv|^2\left(-\sum_{d=1}^{3}\frac{1}{\Delta x_d}
\left(F^n_{\bm{j}+e_d/2}- F^n_{\bm{j}-e_d/2}\right)
-\sum_{d=1}^{3}\frac{1}{\Delta x_d}
\left(R^{n-}_{\bm{j}+e_d/2}- R^{n+}_{\bm{j}-e_d/2}\right)\right) \dd \bv.
\end{equation}
}
Summing \eqref{eq:proof_energy_2} for all $\bj\in\bJ$, together with periodic boundary conditions in $\bx$ and the fact that the regularization term only revises the moment coefficients at the highest order and does not change the total energy, we have
\begin{equation}
\label{eq:proof_energy_3}
  \mE_{K,h}^{n+1, \ast} \triangleq \frac{\Delta V}{2} \sum_{\bj \in \bJ}\int_{\bbR^3} |\bv|^2 f_{\bj}^{n+1, \ast} \dd \bv = \frac{\Delta V}{2} \sum_{\bj \in \bJ}\int_{\bbR^3} |\bv|^2 f_{\bj}^{n} \dd \bv \triangleq \mE_{K,h}^{n}. 
\end{equation}
For the Lorentz force step, \eqref{eq:full_B1} at time level $n+1$ is 
\begin{equation}
  \label{eq:proof_energy_4}
  \frac{ \bm{B}^{n+3/2}_{\bm{j}} - \bm{B}^{n+1}_{\bm{j}} }
  {\Delta t / 2}
  =- \Pi_{\bx} \bE_{\bj}^{n+1},
\end{equation}
which, combining with \eqref{eq:full_B2}, yields
\begin{equation}
  \label{eq:proof_energy_5}
  \frac{ \bm{B}^{n+3/2}_{\bm{j}} - \bm{B}^{n+1/2}_{\bm{j}} }
  {\Delta t}
  =- \Pi_{\bx} \bE_{\bj}^{n+1}, \qquad \frac{ \bm{B}^{n+1/2}_{\bm{j}} - \bm{B}^{n-1/2}_{\bm{j}} }
  {\Delta t}
  =- \Pi_{\bx} \bE_{\bj}^{n}. 
\end{equation}
Similarly as in the proof of Theorem \ref{thm:energy_2}, summing up \eqref{eq:proof_energy_5}
for $\bj$, we obtain 
\begin{equation}
  \label{eq:proof_energy_6}
\begin{aligned}
  &(\mE^{n+1}_{K, h} + \mE^{n+1}_{E, h} + \mE^{n+1}_{B, h}) - (\mE^{n+1, \ast}_{K, h} + \mE^{n}_{E, h} + \mE^n_{B, h})\\ 
  & \qquad =\Delta t \Delta V \sum_{\bj \in \bJ} \left[\Pi_{\bx} \bB^{n+1/2}_{\bj} \cdot (\bE^{n+1}_{\bj} + \bE^{n}_{\bj})
  -\Pi_{\bx} (\bE^{n+1}_{\bj} + \bE^{n}_{\bj})\cdot
  \bB^{n+1/2}_{\bj} \right]\\
  & \qquad =\Delta t \Delta V \sum_{\bj \in \bJ} \Pi_{\bx} \cdot \left(\bB^{n+1/2}_{\bj} \cdot
  (\bE^{n+1}_{\bj} + \bE^{n}_{\bj})\right)=0,
  \end{aligned}
\end{equation}
with
\begin{equation}
  \label{eq:proof_energy_7}
   \mE^{n}_{E, h} =\frac{\Delta V}{2} \sum_{\bj\in \bJ} |\bE_{\bj}^n|^2, \qquad \mE^{n}_{B, h} =\frac{\Delta V}{2} \sum_{\bj\in \bJ} \bB_{\bj}^{n+1/2} \cdot \bB_{\bj}^{n-1/2},
\end{equation}
where in the last equality of \eqref{eq:proof_energy_6} we use the definition \eqref{eq:operator_Pi} and  periodic boundary conditions in $\bx$.
Then the proof is completed by collecting \eqref{eq:proof_energy_3} and \eqref{eq:proof_energy_6}.
\end{proof}

Similar to the discussion in Sec. \ref{sec:semi}, the fully discrete scheme with time integrator {\bf Scheme-I} exactly preserves the total energy, while the fully discrete scheme with time integrator {\bf Scheme-II} achieves near conservation of the total energy.
On the other hand, the computation of {\bf Scheme-I} is demanding and requires to invert  a nonlinear  coupled system of $(\bE,\bB,\bu)$, while {\bf Scheme-II} advances $\bB$ explicitly and only deals with a smaller linear system of $(\bu,\bE)$. Thus we use {\bf Scheme-II} in the numerical simulation, which is more efficient for those benchmark examples compared to {\bf Scheme-I}.


\section{Numerical experiments}
\label{sec:num}
In this section, we present numerical results to demonstrate the performance of the proposed scheme {\bf Scheme-II} for several benchmark examples under 1D2V and 2D3V settings. All the results are computed on the model named Intel(R) Xeon(R) Gold 5218 CPU @ 2.30GHz with the technique of multi-thread adopted. We set $\text{CFL}=0.1$ unless otherwise specified. For conservation properties of all tests, we examine the following two measures:
\begin{itemize}
	\item relative error in the mass:
	\begin{equation}
		\label{eq:error_P}
		\mV(\mP_h^n) = |\mP_h^n - \mP_h^0| / \mP_h^0,
	\end{equation}
\item relative error in the total energy
 \begin{equation}
	\label{eq:var_energy}
	\mV(\mE_{{\rm total}, h}^n) = |\mathcal{E}^n_{{\rm total}, h}-\mathcal{E}^0_{{\rm total}, h}|/|\mathcal{E}^0_{{\rm total}, h}|.
\end{equation}
\end{itemize}


We begin with the VM system \eqref{eq:vlasov} and \eqref{eq:maxwell} in a simple 1D2V setting, which becomes
\begin{subequations}
\label{eq:vm1d2v}
\begin{align}
   	\dfrac{\partial f}{\partial t}&+v_1\dfrac{\partial f}{\partial x}+ (E_1 + v_2 B_3) \pd{f}{v_1} + (E_2 - v_1 B_3) \pd{f}{v_2} = 0,\\
		\dfrac{\partial B_3}{\partial t}&=-\dfrac{\partial E_2}{\partial x}, \qquad \\
 		\dfrac{\partial E_1}{\partial t}&=-\int_{\mathbb{R}^2}v_1 f(t,x,\bm{v})\dd \bm{v},\qquad \\
 		\dfrac{\partial E_2}{\partial t}&=-\dfrac{\partial B_3}{\partial x} -\int_{\mathbb{R}^2}v_2 f(t,x,\bm{v})\dd \bm{v}. 
 		\end{align}
\end{subequations}
Here, $f=f(t,x,v_1,v_2)$, $\bE(t,x)=(E_1(x,t),E_2(x,t),0)$ and $\bB=(0,0,B_3(t,x))$ where $x\in\Omega\subset\mathbb{R}$ and $(v_1,v_2)\in \mathbb{R}^2$. $f, \bB, \bE$ are periodic in the $x$-direction. More details can be found in \cite{califano1998kinetic,Hamilton}.
For this simplified model, considering Theorem \ref{thm:energy_2}, the discrete electric energy, magnetic energy, and kinetic energy are reduced into 
\begin{subequations}
	\label{eq:E_1D2D}
	\begin{align}
  	\mathcal{E}^n_{E,h} &= \frac{\Delta x}{2}\sum^{N}_{j=1} [(E^n_{1, j})^2+(E^n_{2, j})^2], \qquad 
  		\mathcal{E}^n_{B,h}=\frac{\Delta x}{2}\sum^{N}_{j=1} B^{n+1/2}_{3, j}B^{n-1/2}_{3, j}, \\
   \mathcal{E}^n_{K,h} &= \frac{\Delta x}{2}\sum_{j=1}^{N}\left[
 \rho^n_j(u^n_{1,j})^2 +\rho^n_j(u^n_{2,j})^2+ 2\rho^n_j\mathcal{T}^n_j\right],
\end{align}
\end{subequations}
with $N$ being the number of meshes in the spatial domain and $\Delta x$ being the spatial mesh size. 

\subsection{Linear Landau damping}
\label{sec:landau}
In this section, we consider the linear Landau damping problem, which was first introduced in \cite{landau196561} and was verified later by experiments in \cite{malmberg1964collisionless}. It refers to the phenomenon that the amplitude of a wave decreases due to the interaction between the particles and the wave. The governing equations of this problem are a VA system \eqref{eq:vp} in the 1D2V setting, which can be written in the simple form of 
the VM system \eqref{eq:vm1d2v} with $\bB=\bz$. The initial condition is given by
\begin{align}
  \label{eq:ex1_ini}
	f(0,x,\bm{v})=\frac{1}{2\pi}e^
	{-\frac{|\bm{v} - \bu|^2}{2\mT}}(1+A\cos(k x)),
	\qquad (x,\bm{v})\in[0,L]\times\mathbb{R}^2,
\end{align} 
where $\bv=(v_1,v_2)$, $A$ is the perturbation amplitude, $k$ is the wave number, and
\begin{equation}
  \bu = \bz, \qquad \mT = 1, \qquad L = \frac{2\pi}{k}.
\end{equation}

We perform numerical simulations with 
$N = 500$ meshes in the spatial domain, and the truncation order of Hermite series set as $M = 30$. We consider two cases in the initial condition \eqref{eq:ex1_ini} with $k = 0.3$, $A = 10^{-5}$ and
$k = 0.4$, $A = 0.01$. 

We first verify the conservation properties of the proposed method.
Fig. \ref{fig:mass_landau} shows the time evolution of the relative error in the mass $\mV (P^n_h )$. Fig. \ref{fig:totalenergy_landau} shows the time evolution of the relative error in the total energy $\mathcal{V}(\mE_{{\rm total}, h}^n)$. It can be observed that the errors stay small, below $10^{-15}$ for both the mass and the total energy, which reflects the mass- and energy-preserving properties of the numerical scheme as illustrated in Theorem \ref{thm:mass_full} and Theorem \ref{thm:energy_3}.
\begin{figure}[!htbp]
	\centering
	\subfigure[$k=0.3$, $A=10^{-5}$.]{
		\includegraphics[width=0.45\linewidth]{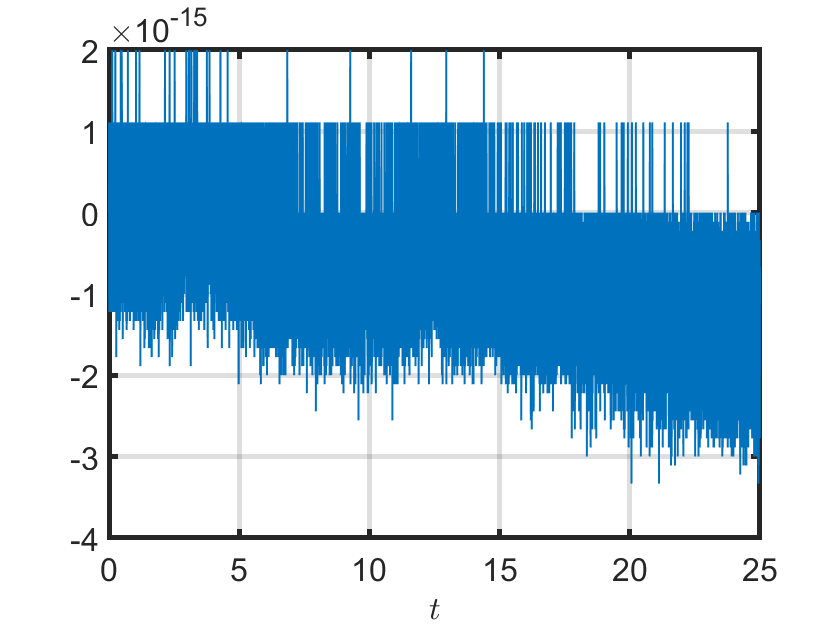}}
	\subfigure[$k=0.4$, $A=0.01$.]{
		\includegraphics[width=0.45\linewidth]{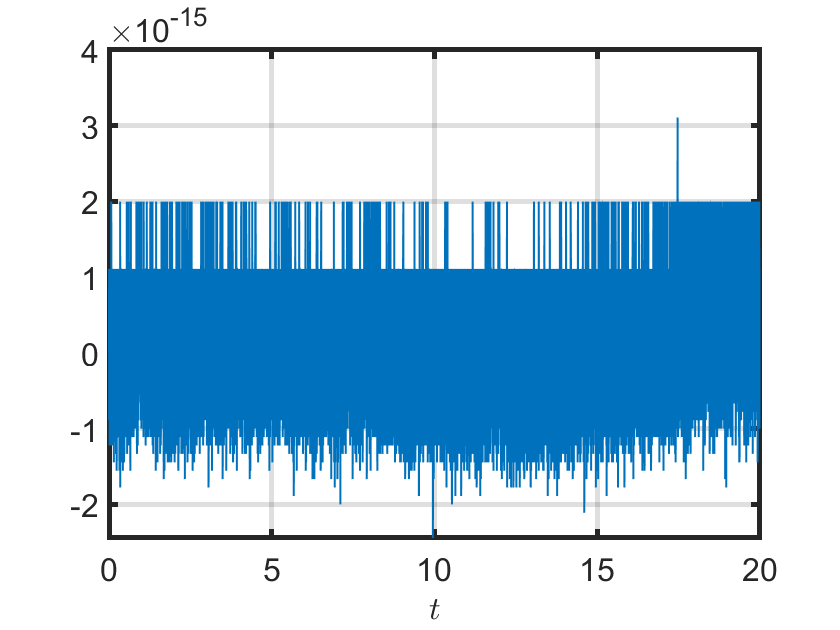}}
	\caption{Time evolution of the relative error in the mass for the Landau damping test in Sec. \ref{sec:landau}. $x$-axis: time $t$; $y$-axis: $\mV (\mP^n_h )$ defined in \eqref{eq:error_P}.}
		\label{fig:mass_landau}
\end{figure}

\begin{figure}[!htbp]
	\centering
	\subfigure[$k=0.3$, $A=10^{-5}$.]{
		\includegraphics[width=0.45\linewidth]{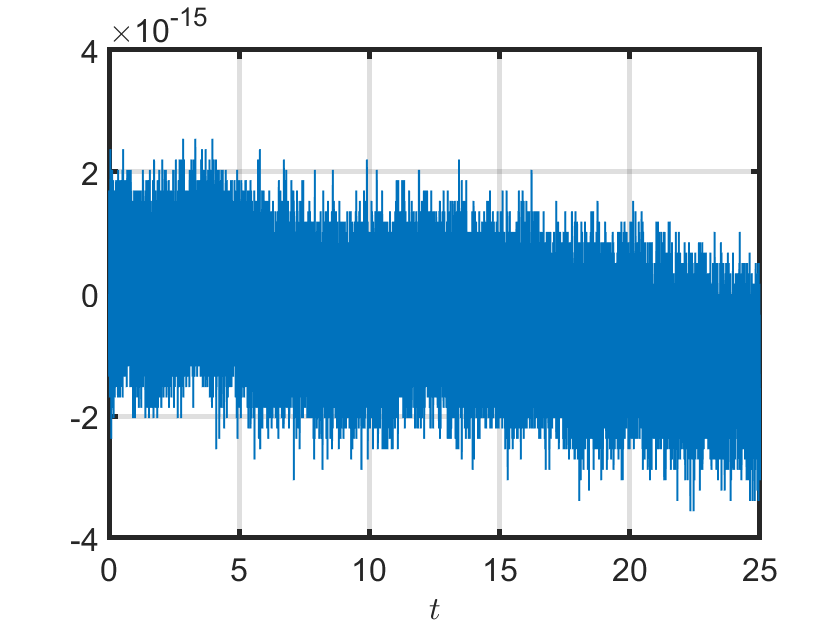}}
	\subfigure[$k=0.4$, $A=0.01$.]{
		\includegraphics[width=0.45\linewidth]{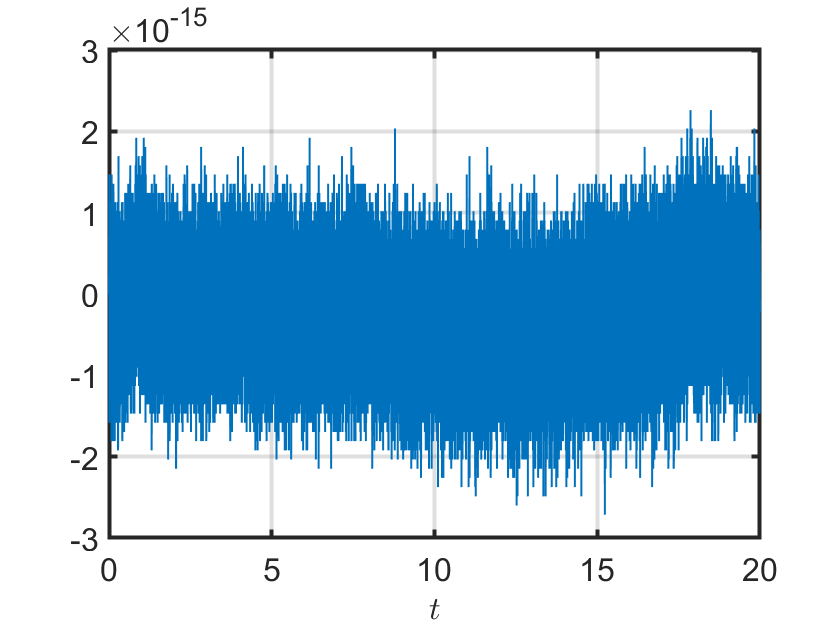}
	}
	
	\caption{Time evolution of the relative error in the total energy for the Landau damping test in Sec. \ref{sec:landau}. $x$-axis: $t$; $y$-axis: $\mV(\mE_{{\rm total}, h}^n)$ defined in \eqref{eq:var_energy}.
		\label{fig:totalenergy_landau}}
\end{figure}

We further collect some sample numerical data to benchmark our schemes. It is known that in Landau damping theory, a large number of slow particles absorb energy from the wave while relatively fewer particles transfer energy to the wave, resulting in particles and the wave tending to synchronize \cite{doveil2005experimental}. Thus we investigate the time evolution of the electric energy $\mathcal{E}_{E}$ which is expected to decay exponentially with a fixed rate $\gamma_L$ related to the wave number $k$, see e.g. \cite{Vlasov2012} for more details. 
We show the time evolution of the electric energy $\mE^n_{E,h}$ in the log scale in Fig. \ref{fig:landau}.
We capture local peak values of the electric energy $\mE^n_{E,h}$, and obtain the
damping slope $\gamma_{L, h}$ using the least-square approximation as in
\cite{Vlasov2012}. Fig. \ref{fig:landau} shows that the numerical
damping slope agrees well with the theoretical result \cite{Hamilton}. Tab. \ref{tab:slope} lists the qualitative results of the damping slope. We observe that the error between the theoretical value $\gamma_L$ and the numerical solution $\gamma_{L, h}$ is quite small compared with the corresponding wave number $k$.

\begin{figure}[!htbp]
	\centering
	\subfigure[$k=0.3$, $A=10^{-5}$.]{
		\includegraphics[width=0.45\linewidth]{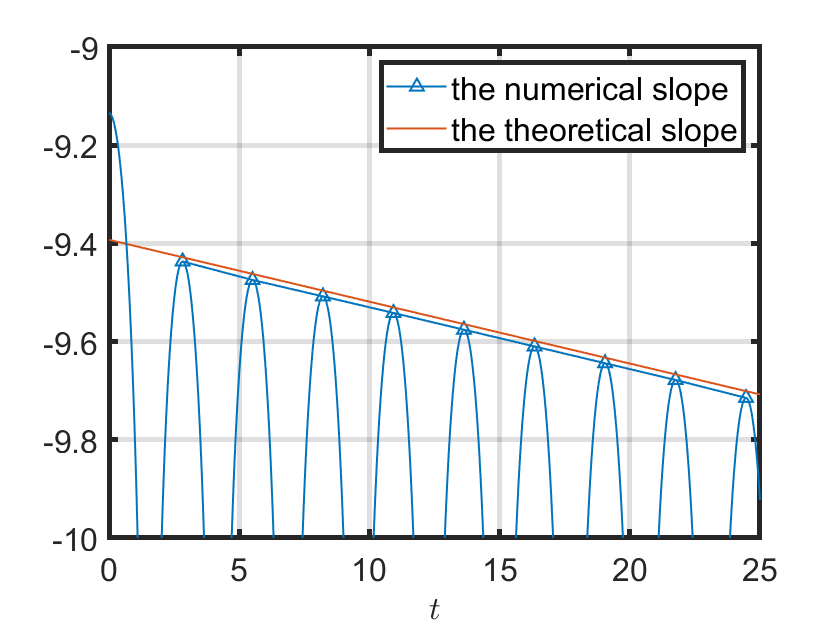}}
	\subfigure[$k=0.4$, $A=0.01$.]{
		\includegraphics[width=0.45\linewidth]{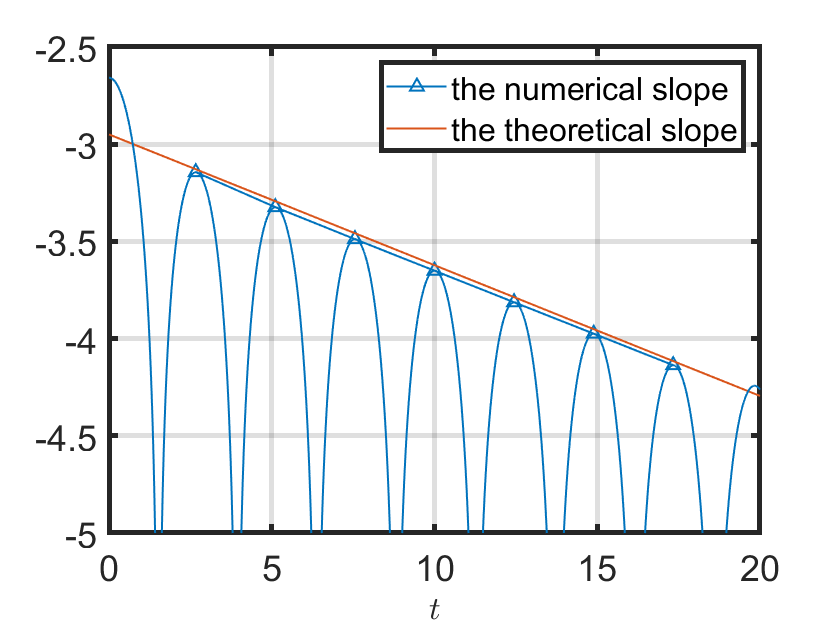}
	}
	\caption{Time evolution of the 
   electric energy for the Landau damping test in Sec. \ref{sec:landau}. $x$-axis: $t$; $y$-axis: 
   $(\log(2\mE_{E,h}^n))/2$ with $\mE_{E,h}^n$ defined in \eqref{eq:E_1D2D}.}
	\label{fig:landau}
\end{figure}

\begin{table}[!htbp]
	\centering
	  \caption{The theoretical and numerical damping slopes with different $k$.}
	\label{tab:slope}
	\smallskip
	\def\arraystretch{1.5}
	{\footnotesize
	\begin{tabular}{c|c|c|c}
		\hline
		Wave number $k$ & Theoretic slope $\gamma_{L}$ & Numerical slope $\gamma_{L,h}$ &$|\gamma_{L}-\gamma_{L,h}|$ \\
		\hline
		0.3 & -0.0126 & -0.0127&0.0001 \\
		\hline
		0.4 & -0.0661 & -0.0673 &0.0012\\
		\hline
	 \end{tabular}
	 }
 
\end{table}

%

\subsection{Two-stream instability}	
\label{sec:twostream}
In this section, we consider the two-stream instability, which is a widespread and simple electrostatic micro-instability phenomenon \cite{stix1992waves} in plasma physics, where the movement of the particles is disturbed when the charged particle beam passes through the plasma, and further generating electric field of space charges, which encourages clustering, and leads to the dual-current instability.
The governing equations of this problem are the simplified VM system \eqref{eq:vm1d2v} in the 1D2V setting, with the initial conditions under the same settings as in \cite{Hamilton}, given by
\begin{align}
 \label{eq:ini_ex2}
	f(0,x,\bv)&=\frac{\rho}{2\pi \mathcal{T}}\left[\frac{1}{2}\exp\left({-\frac{|\bv - \bu|^2}{2\mathcal{T}}}\right)+\frac{1}{2}\exp\left(-\frac{|\bv + \bu|^2}{2\mathcal{T}}\right)\right],\\[2mm]
	\label{eq:ex2_E_B}
  \bE(0,x) &= \bz, \qquad B_3(0, x) = A\sin(kx),\qquad (x,\bm{v})\in[0,L]\times\mathbb{R}^2,
\end{align}
where $L = 2\pi,\ A = 10^{-3}$ and $k = 1$. At $t=0$, we set $\rho=1, \ \bu = (0.2, 0)$, and $\mathcal{T}=10^{-3}.$
Note here that for this problem involving the magnetized plasma, the initial magnetic field perturbation is also considered besides the regular perturbation of the probability density distribution to drive instability.


We perform numerical simulations with 
$N = 200$ meshes in the spatial domain, and the truncation order of Hermite series taken as $M = 30$.
We first show the time evolution of the relative error in the mass $\mV(\mP_h^n)$ and the relative error in the 
 total energy $\mV(\mE_{{\rm total}, h}^n)$ in Fig. \ref{fig:twostream_var_P} and \ref{fig:twostream_var_E}, respectively. It can be observed that the errors are below $10^{-14}$, which demonstrate excellent conservation properties of the numerical scheme as illustrated in Theorem \ref{thm:energy_3} and Theorem \ref{thm:mass_full}. We also take a closer look at the time evolution of the errors between the electromagnetic and kinetic energy and their initial counterpart in Fig. \ref{fig:twostream_var_EB}. It shows the transference of total energy between the kinetic energy and the electromagnetic energy, which is consistent with the total energy conservation.
%
\begin{figure}[!htbp]
	\centering
	\subfigure[$\mV(\mP_h^n)$]{
		\includegraphics[width=0.3\linewidth, height=0.3\linewidth]{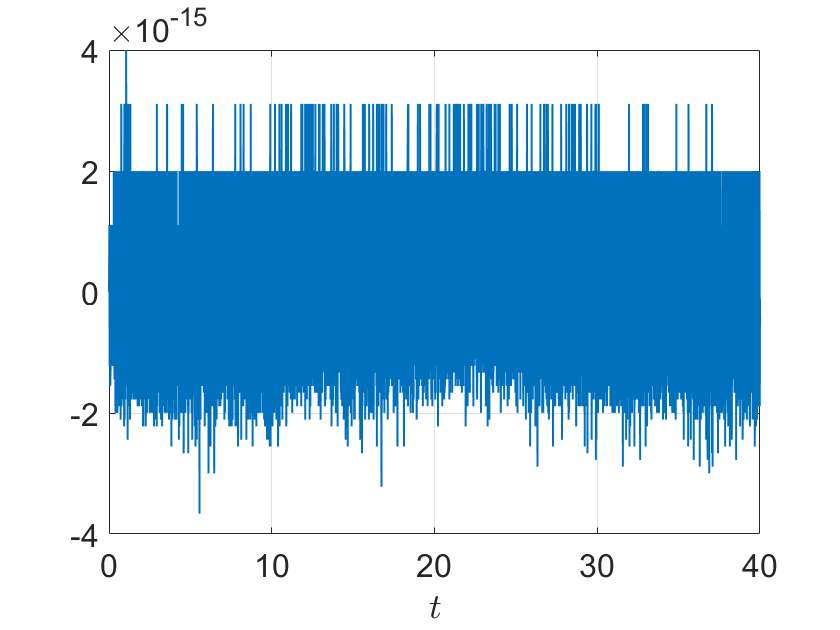}
		\label{fig:twostream_var_P}
	} 
	\subfigure[$\mV(\mE^n_{{\rm total},h})$]{
		\includegraphics[width=0.3\linewidth, height=0.3\linewidth]{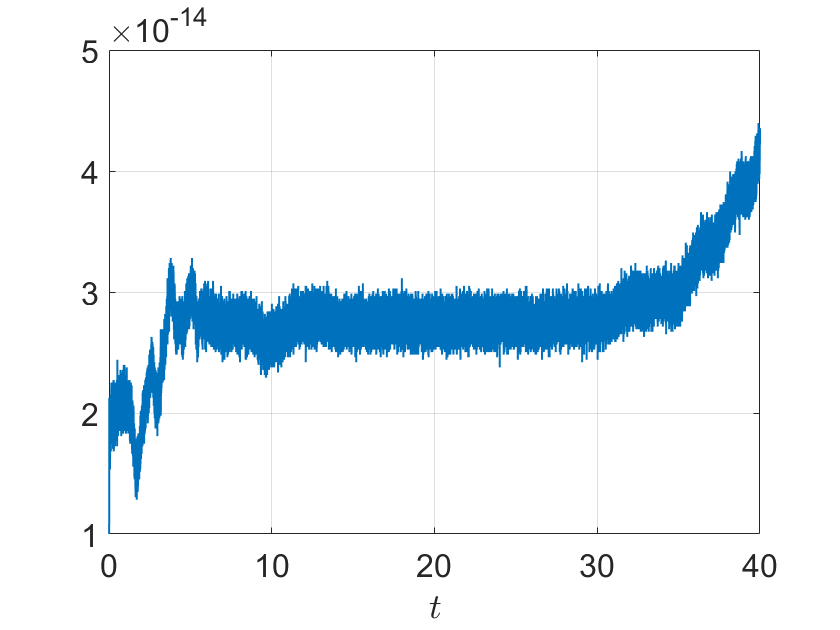}
		\label{fig:twostream_var_E}
	} 
	\subfigure[error of kinetic and electromagnetic energy]{
		\includegraphics[width=0.3\linewidth, height=0.3\linewidth]{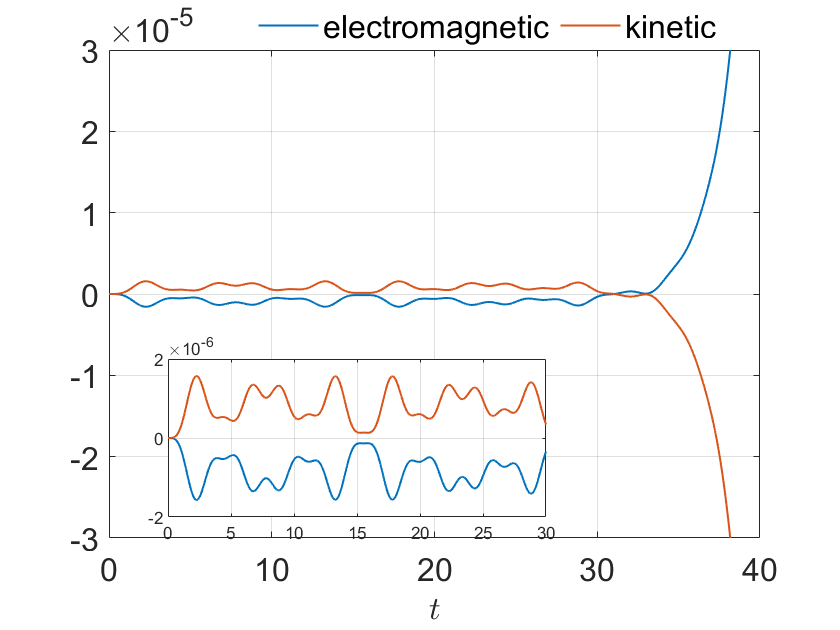}
		\label{fig:twostream_var_EB}
	}
\caption{Two-stream instability in Sec. \ref{sec:twostream}. Time evolution of the relative error in the mass $\mV(\mP_h^n)$ (left), the relative error in the total energy $\mV(\mE_{{\rm total}, h}^n)$ (middle) and the error in  the electromagnetic energy
	$(\mE^n_{B,h} +\mE^n_{E,h}) - (\mE^0_{B,h} +\mE^0_{E,h})$
	and the kinetic energy $\mE^n_{K,h} - \mE^0_{K,h}$ (right). $x$-axis: $t$; $y$-axis: the corresponding error. }
	\label{fig:twostream_E}
\end{figure}


We also plot the time
evolution for the electric, magnetic and kinetic energy defined in
\eqref{eq:E_1D2D} with difference spatial sizes with $N = 200, 400, 800$ and $1000$. In Fig. \ref{fig:twostream_sol_energy_change}, it shows that the numerical solution with $N = 200$ is indistinguishable compared with $N=1000$. The solution is well resolved even with $N=200$.
Fig. \ref{fig:twostream_sol_energy} shows the numerical solution obtained by the proposed numerical scheme with $N=200$, which matches well with the reference solution obtained by the discrete velocity method (DVM).

\begin{figure}[!htbp]
	\centering
	\subfigure[comparison among different spatial sizes]{
		\includegraphics[width=0.42\linewidth]{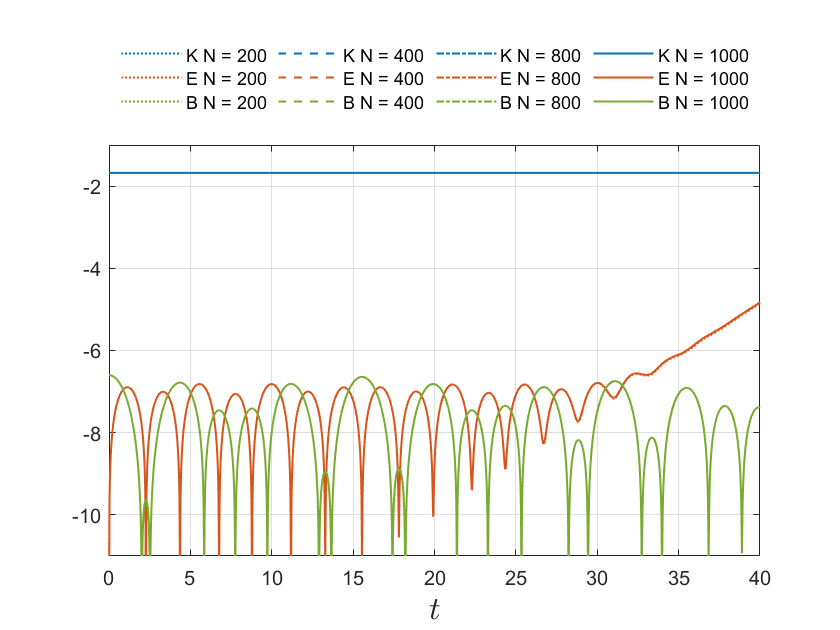}
		\label{fig:twostream_sol_energy_change}
	}\quad 
	\subfigure[comparison with DVM]{
		\includegraphics[width=0.42\linewidth]{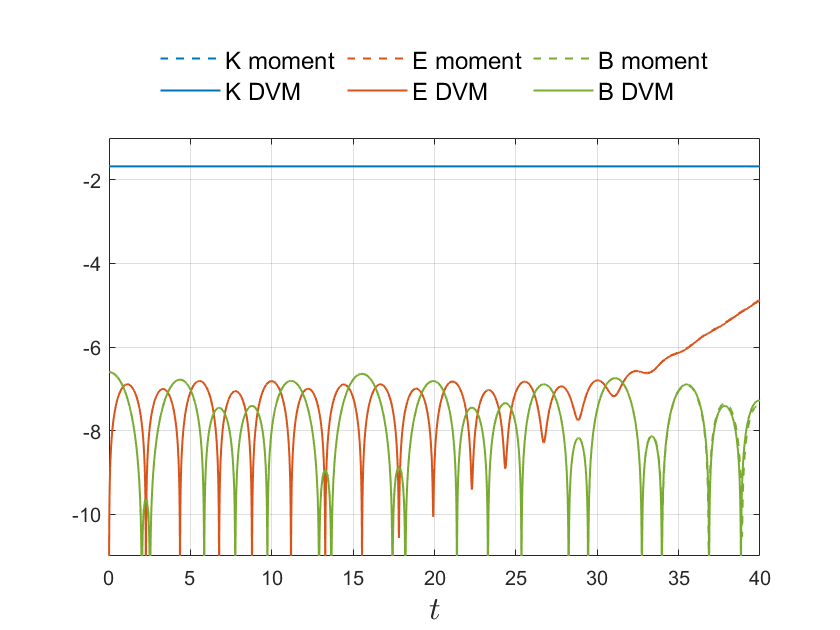}
		\label{fig:twostream_sol_energy}
	}
	\caption{Two-stream instability in Sec. \ref{sec:twostream}. Time evolution of 
		the kinetic energy $\mE^n_{K,h}$, electric energy $\mE^n_{E,h}$, and 
		the magnetic energy $\mE^n_{B,h}$. Here $y$-axis denotes the logarithmic form of the energies
		$\log_{10}(\mathcal{S}/L)$ with $\mathcal{S} = \mE^n_{K,h},
		\mE^n_{E,h}, \mE^n_{B,h} $. (a) The comparison between the proposed scheme (dashed line) with the reference solution (solid line) obtained by DVM. (b) The proposed scheme with different spatial sizes. 
		}
	\label{fig:twostream_sol}
\end{figure}

We further investigate the marginal distribution defined as
\begin{equation}
  \label{eq:MDF}
  g(t, x, v_1) = \int_{\bbR} f(t, x, \bv) \dd v_2.
\end{equation}
 Fig. \ref{fig:twostream_sol_dis} shows the marginal distribution function at $t=0$ and $t=40$. For the initial distribution, there exist two peaks. As time evolves, the oscillations appear gradually, which can be
clearly seen at $t = 40$. 
For a more detailed visualization, we show the time evolution of the components of the electric and
magnetic field by the proposed scheme as well as the DVM in Fig. \ref{fig:twostream_EB}. Clearly, the results obtained by our numerical method agree well with those of the
DVM. Moreover, it can be observed that for the component $E_1$, it is
quite small at the beginning, and then increases to some periodic
structures from $t = 30$; For $E_2$, it has the $\cos$-type structure
and then gradually changes while preserving the periodic structure; For $B_3$, it has the $\sin$-type structure at the beginning, and
is evolving with this structure kept. 
\begin{figure}[!htbp]
	\centering
	\subfigure[$t=0$]{
		\includegraphics[width=0.45\linewidth]{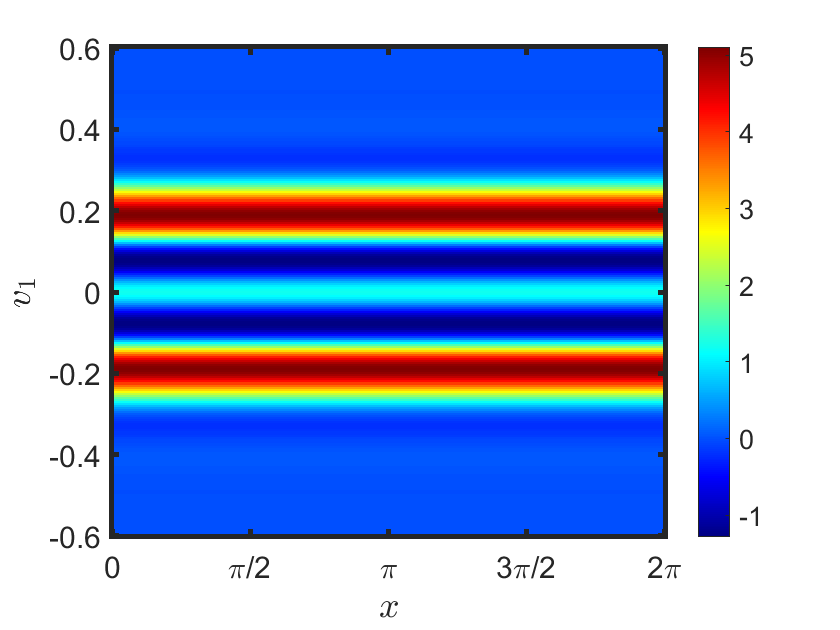}
		\label{fig:twostream_sol_f_v1x_t0}
	}\quad 
	\subfigure[$t=40$]{
		\includegraphics[width=0.45\linewidth]{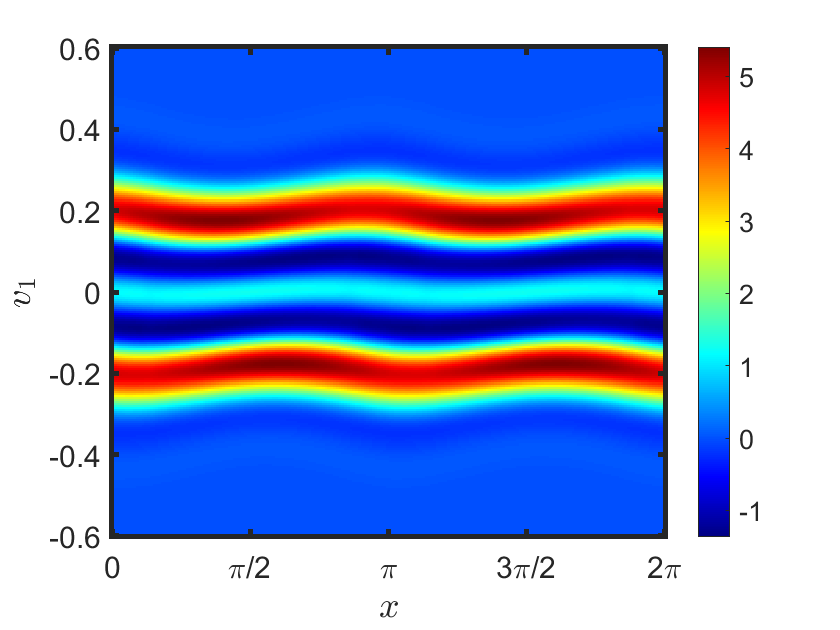}
		\label{fig:twostream_sol_f_v1x_t40}
	}
	\caption{Two-stream instability in Sec. \ref{sec:twostream}. The marginal distribution
		function $g(t, x, v_1)$ at $t = 0$ (left) and $40$ (right). }
	\label{fig:twostream_sol_dis}
\end{figure}
\begin{figure}[!htbp]
	\centering
	\subfigure[$E_1$, moment]{
		\includegraphics[width=0.45\linewidth]{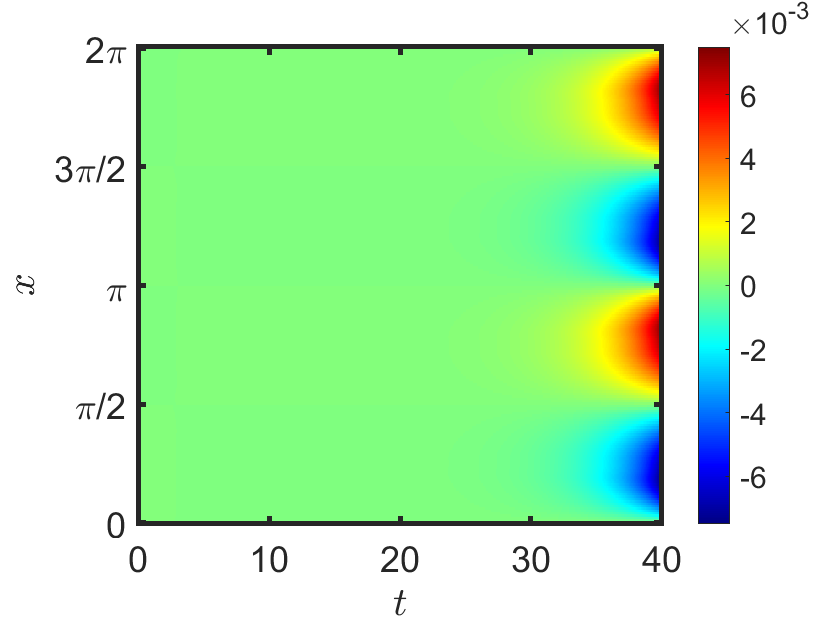}
		\label{fig:twostream_sol_E1}
	} \quad 
	\subfigure[$E_1$, DVM]{
		\includegraphics[width=0.45\linewidth]{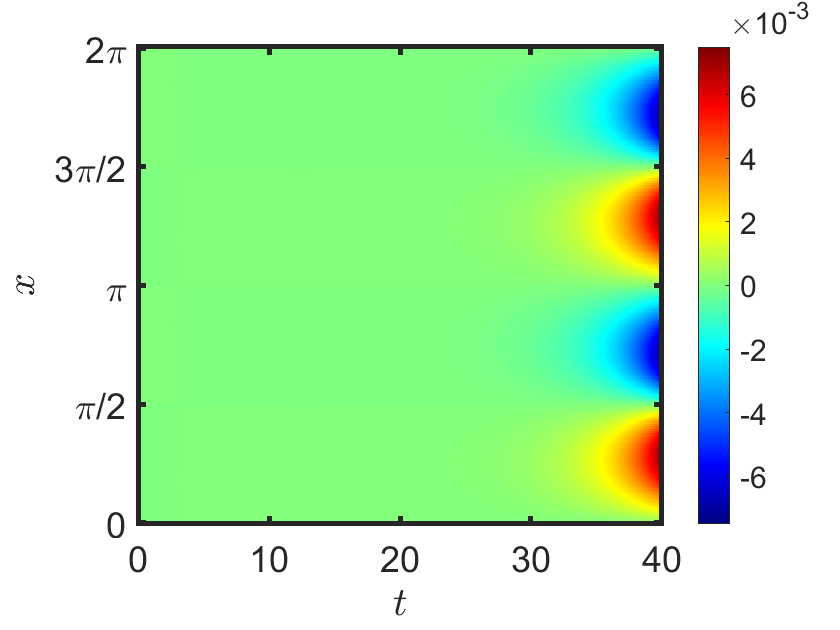}
		\label{fig:twostream_sol_E1_DVM}
	}
	\subfigure[$E_2$, moment]{
		\includegraphics[width=0.45\linewidth]{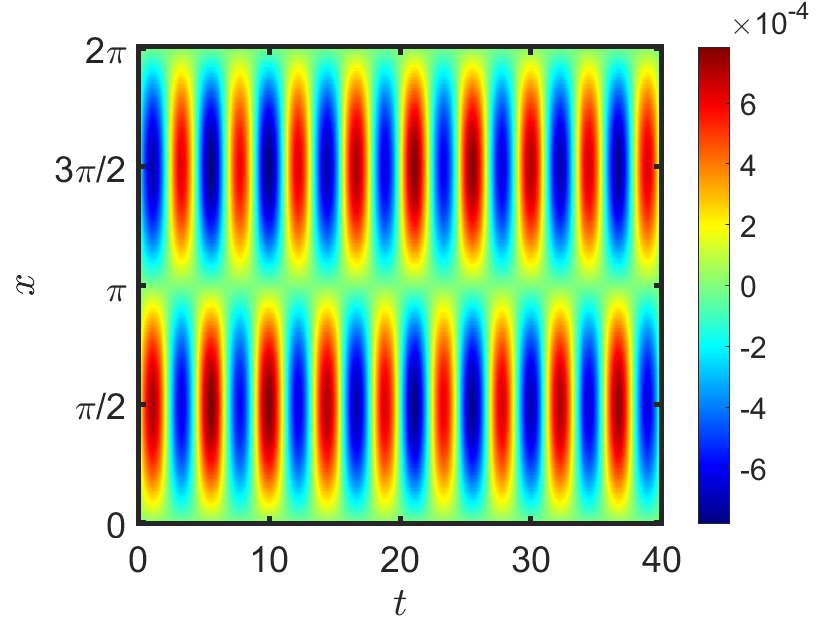}
		\label{fig:twostream_sol_E2}
	} \quad 
	\subfigure[$E_2$, DVM]{
		\includegraphics[width=0.45\linewidth]{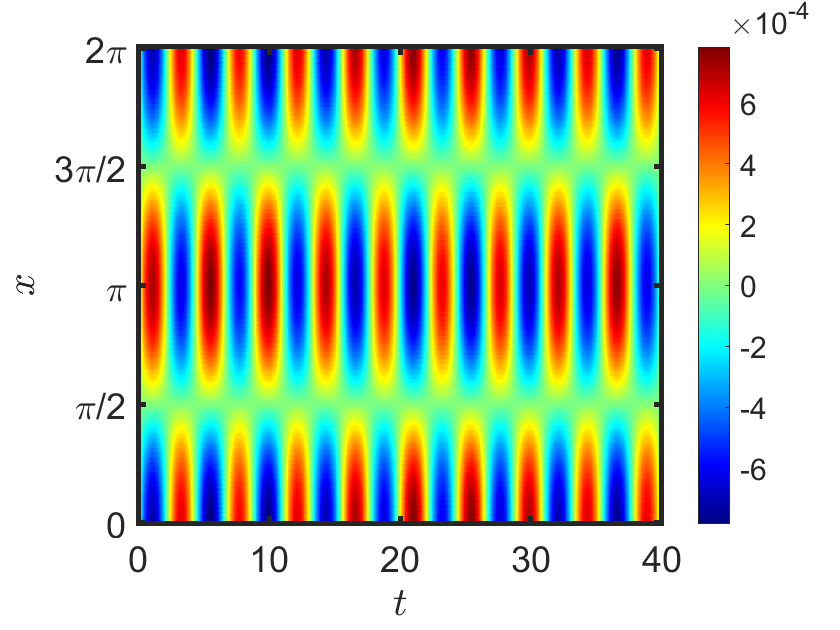}
		\label{fig:twostream_sol_E2_DVM}
	} 
	\subfigure[$B_3$, moment]{
		\includegraphics[width=0.45\linewidth]{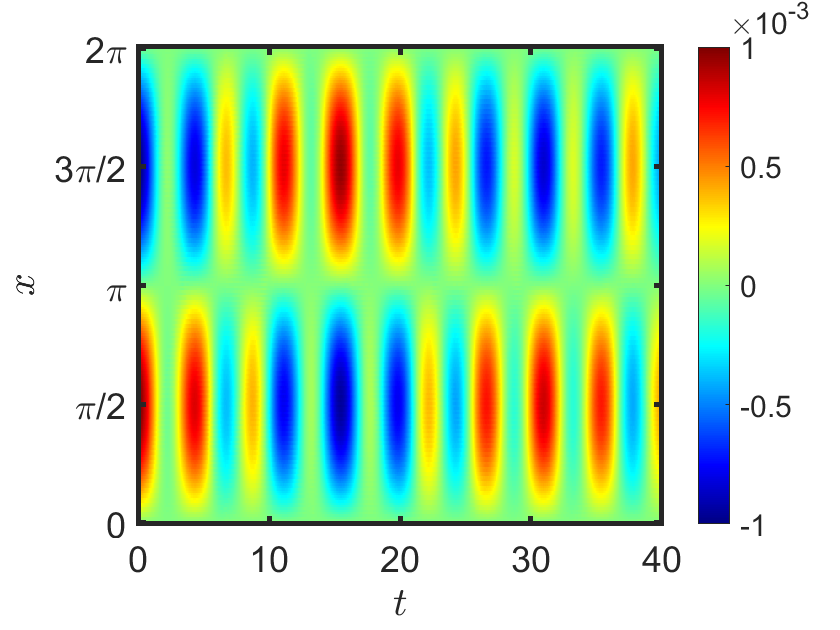}
		\label{fig:twostream_sol_B3}
	}\quad 
	\subfigure[$B_3$, DVM]{
		\includegraphics[width=0.45\linewidth]{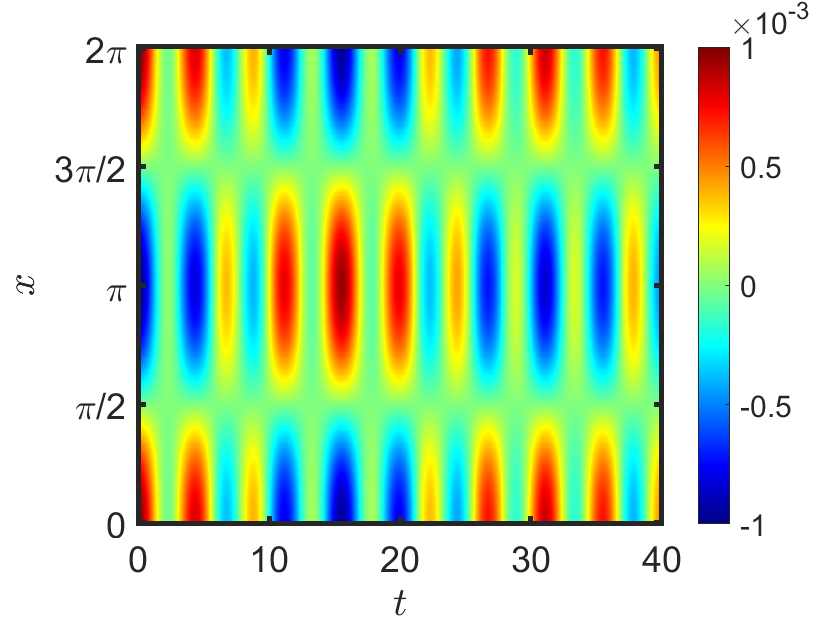}
		\label{fig:twostream_sol_B3_DVM}
	}
	\caption{Two-stream instability in Sec. \ref{sec:twostream}. Time evolution of the
		electric fields $E_1$ (top) and $E_2$ (middle), and the magnetic
		field $B_3$ (bottom), obtained by the proposed numerical scheme (left) and the DVM (right).
		}
	\label{fig:twostream_EB}
\end{figure}

\subsection{Weibel instability}	
\label{sec:weibel}
In this section, we consider the Weibel instability. In plasma physics, when the uniformly distributed electron current sheet is disturbed by the magnetic field, it produces a disturbance velocity. Then, the positive and negative current sheets are partly concentrated, and partly scattered, and generate a disturbance current \cite{Plasma}. According to Faraday's law, the current disturbance in turn generates a new magnetic field. The increased local current density also causes the plasma to be strongly pinched to form high-density filaments \cite{Plasma}. This is the so-called Weibel instability, which is also very common in plasma physics, especially in astrophysics. 
The governing equations of this problem are the simplified VM system \eqref{eq:vm1d2v} in the 1D2V setting, with the initial conditions under the same settings as in \cite{VM_Zhong}, given by
\begin{align}
 \label{eq:ini_weibel}
 	f(0,x,\bv)=&\frac{\rho}{2\pi \mathcal{T}}
 	\left[\frac{1}{6}\exp\left(-\frac{|\bv - \bu_1|^2}{2\mathcal{T}}\right)
 	+\frac{5}{6}\exp\left(-\frac{|\bv - \bu_2|^2}{2\mathcal{T}}\right)\right],\\
 	 \label{eq:weibel_B}
 \bE(0,x)& = \bz,\qquad B_3(0,x)=A\sin(kx),\qquad (x,\bv)\in[0,L]\times\mathbb{R}^2
\end{align}
 with $A=10^{-3}, \ k=0.2, \ L=2\pi/k=10\pi$. At $t=0$,
 $\rho=1,\ \bu_1 = (0, 0.5), \ \bu_2 = (0, -0.1), \ \mathcal{T}(0,x)=5\times10^{-3}.$

\begin{figure}[!htbp]
 	\centering
 	\subfigure[$\mV(\mP_h^n)$ ]{
 		\includegraphics[width=0.3\linewidth, height=0.3\linewidth]{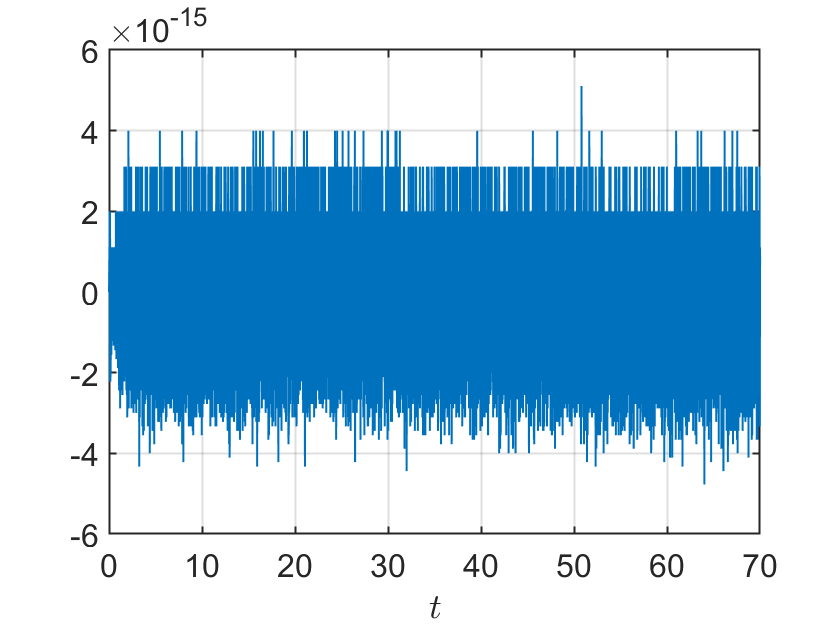}
 	}
  \subfigure[$\mV(\mE_{{\rm total}, h}^n)$ ]{
	 \includegraphics[width=0.3\linewidth, height=0.3\linewidth]{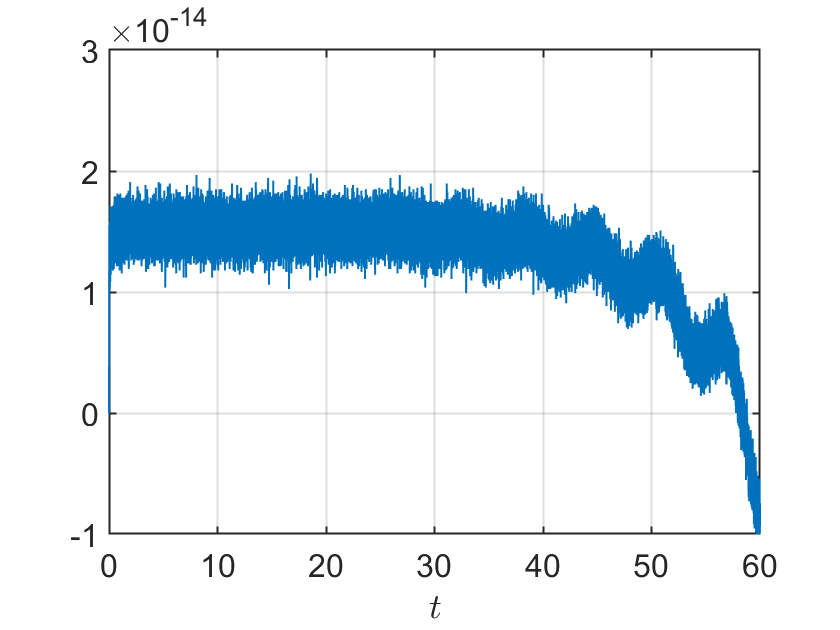}
	 }
  \subfigure[variation of each energy]{
	 \includegraphics[width=0.3\linewidth, height=0.3\linewidth]{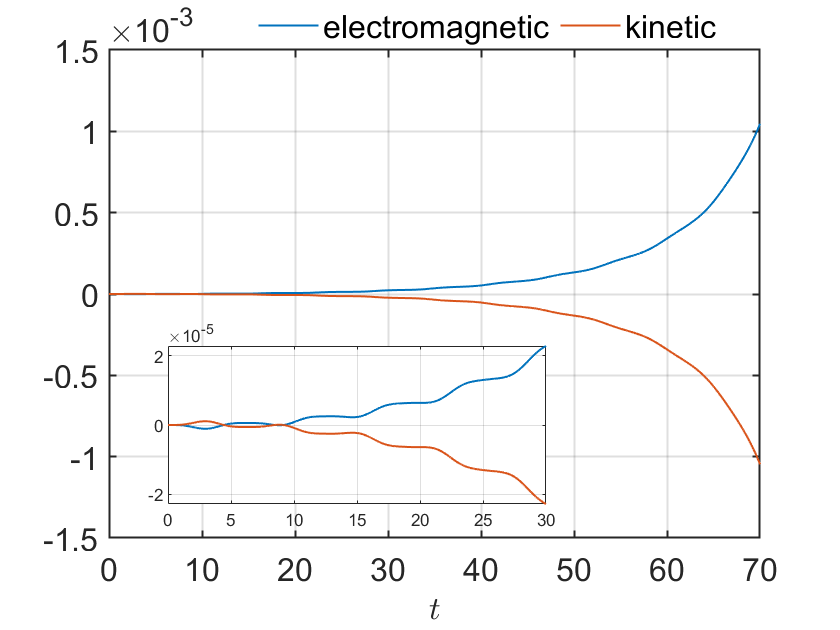}
	 \label{fig:weibel_var_EB}
	 }
 	\caption{Weibel instability in Sec. \ref{sec:weibel}. Time evolution of the relative error in the mass $\mV(\mP_h^n)$ (left). Time evolution of the relative error in the total energy $\mV(\mE_{{\rm total}, h}^n)$ (middle). The errors of the electromagnetic energy  $(\mE^n_{B,h} +\mE^n_{E,h}) - (\mE^0_{B,h} +\mE^0_{E,h})$
   and kinetic energy $(\mE^n_{K,h} - \mE^0_{K,h})$ (right). $x$-axis denotes time $t$, and $y$-axis denotes the corresponding error. }
	\label{fig:weibel_E}
\end{figure}
Numerical simulations are performed with $N=2000$ meshes in the spatial domain and the truncation order of the moment method taken as $M=30$. The time evolution of the relative error in the mass and the total energy is shown in Fig. \ref{fig:weibel_E}. For the Weibel instability, the relative errors of the mass and energy is still quite small, which validates the conservation properties of this numerical scheme. The errors in electromagnetic and kinetic energy are plotted in Fig. \ref{fig:weibel_var_EB}, the behavior of which is similar to that of the two-stream instability.

\begin{figure}[!htbp]
	\centering
	\subfigure[comparison among different spatial sizes]{
		\includegraphics[width=0.45\linewidth]{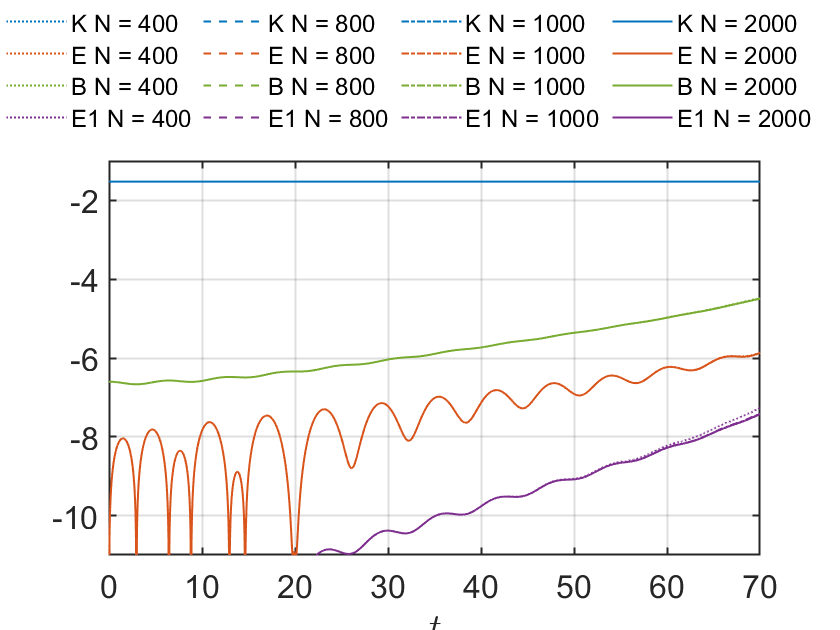}
		\label{fig:weibel_sol_energy_N}
	}
  \subfigure[comparison with the DVM ]{
		\includegraphics[width=0.45\linewidth]{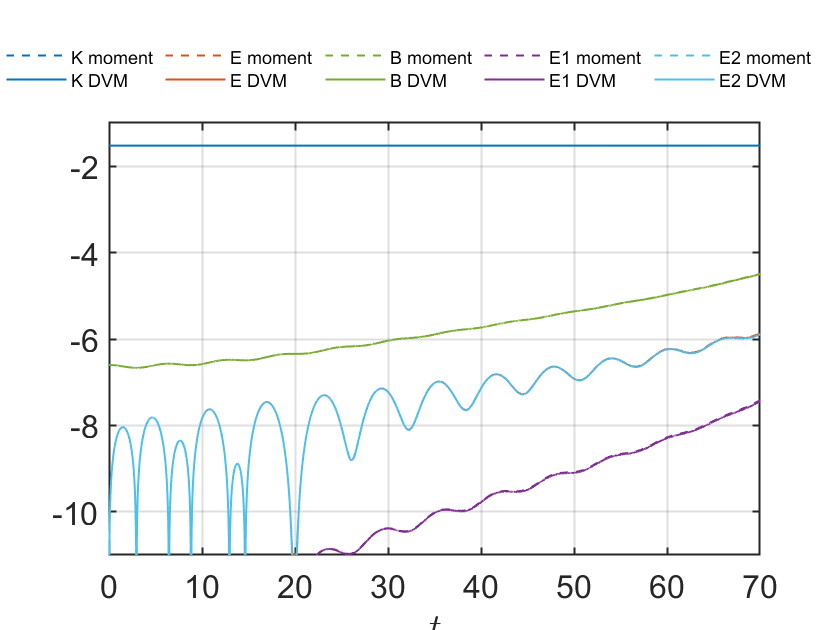}
		\label{fig:weibel_sol_energy}
	}	\quad 
	\caption{Weibel instability in Sec. \ref{sec:weibel}. Time evolution of 
	the kinetic energy $\mE^n_{K,h}$, electric energy $\mE^n_{E,h}$,
	the magnetic energy $\mE^n_{B,h}$ and $\mE^n_{E_i,h},i=1,2$.The $y$-axis denotes the logarithmic form of the energy
   $\log_{10}(\mathcal{S}/L)$ with $\mathcal{S} = \mE^n_{K,h},
   \mE^n_{E,h}, \mE^n_{B,h}, \mE^n_{E_i,h}, i = 1,2$.
	(a) The proposed numerical scheme with different spatial sizes. (b) The comparison between the proposed numerical scheme (dashed line) and the DVM (solid line). 
}
	\label{fig:weibel_sol}
\end{figure}

\begin{figure}[!htbp]
	\centering
  \subfigure[$t=0$]{
		\includegraphics[width=0.45\linewidth]{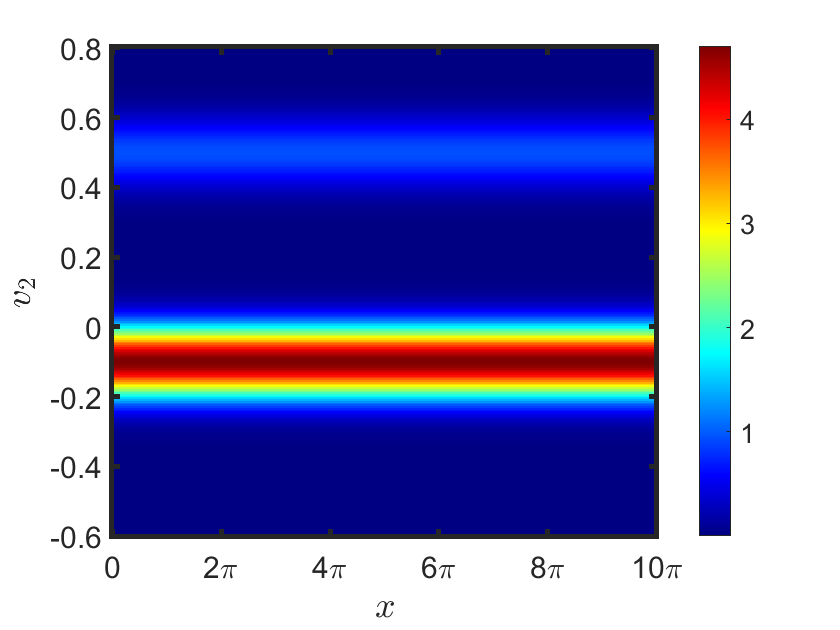}
		\label{fig:weibel_sol_f_v2x_t0}
	} \quad
	\subfigure[$t=30$]{
		\includegraphics[width=0.45\linewidth]{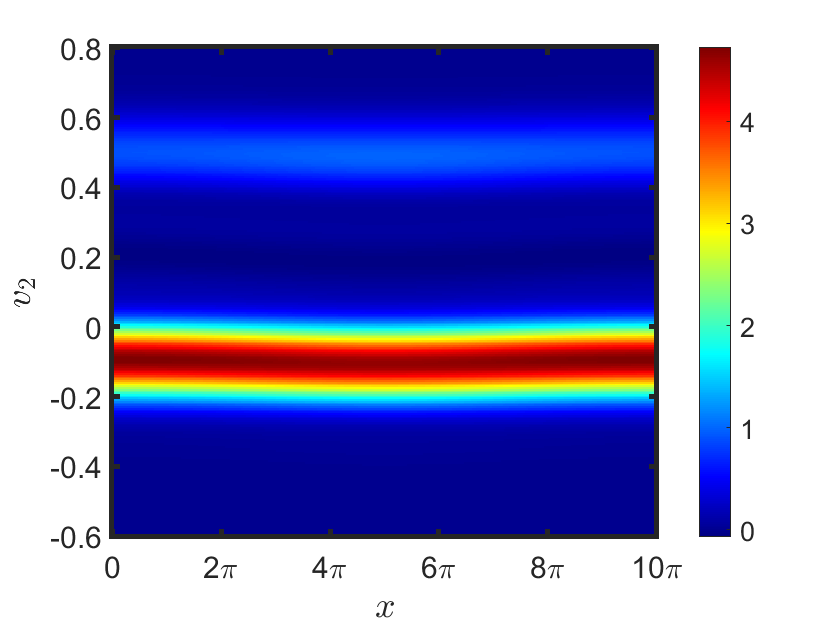}
		\label{fig:weibel_sol_f_v2x_30}
	} \\
  \subfigure[$t=50$]{
	\includegraphics[width=0.45\linewidth]{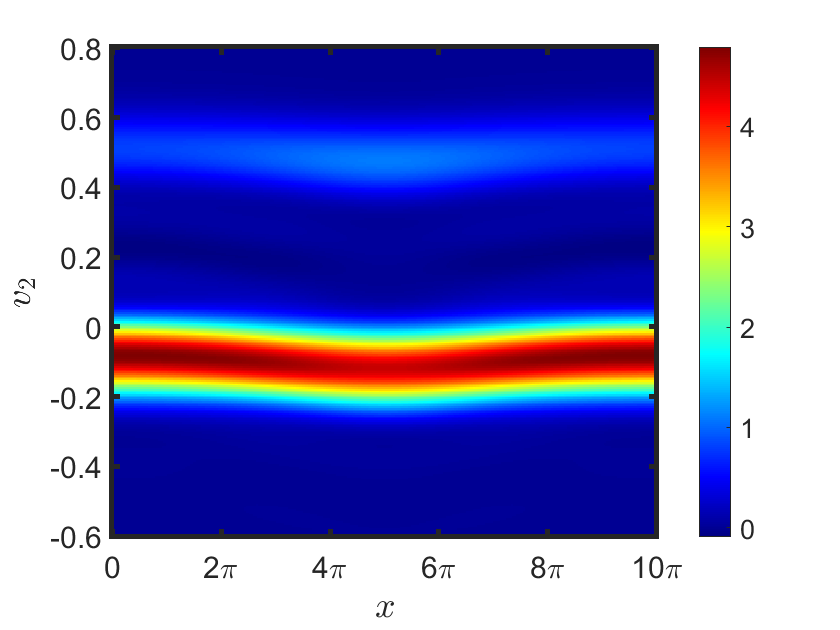}
	\label{fig:weibel_sol_f_v2x_t50}
	} \quad 
	\subfigure[$t=70$]{
		\includegraphics[width=0.45\linewidth]{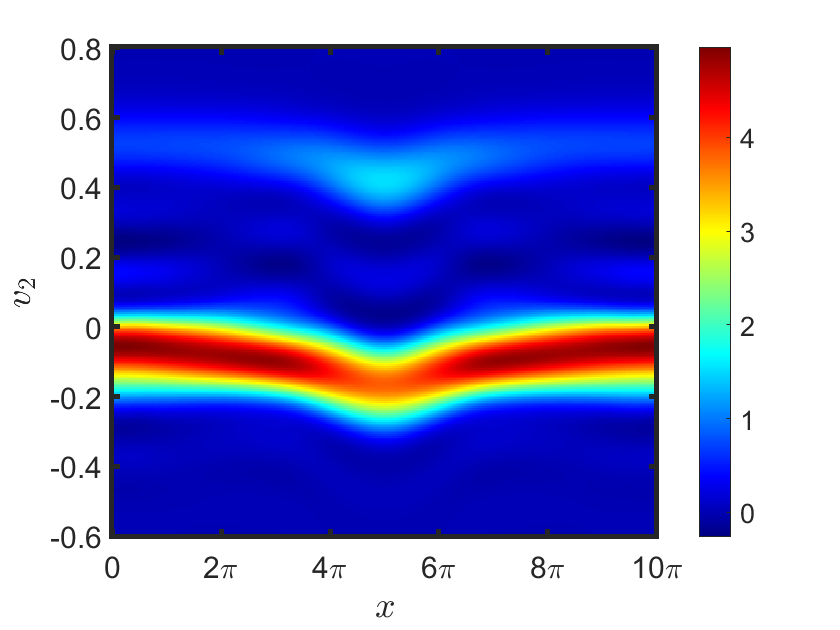}
		\label{fig:weibel_sol_f_v2x_t70}
	}
	\caption{Weibel instability in Sec. \ref{sec:weibel}. The marginal distribution function $g(t, x, v_2)$ at $t = 0, 30, 50$ and $70$. }
	\label{fig:weibel_sol_dis}
\end{figure}

\begin{figure}[!htbp]
	\centering
	\subfigure[$E_1$, moment]{
		\includegraphics[width=0.45\linewidth]{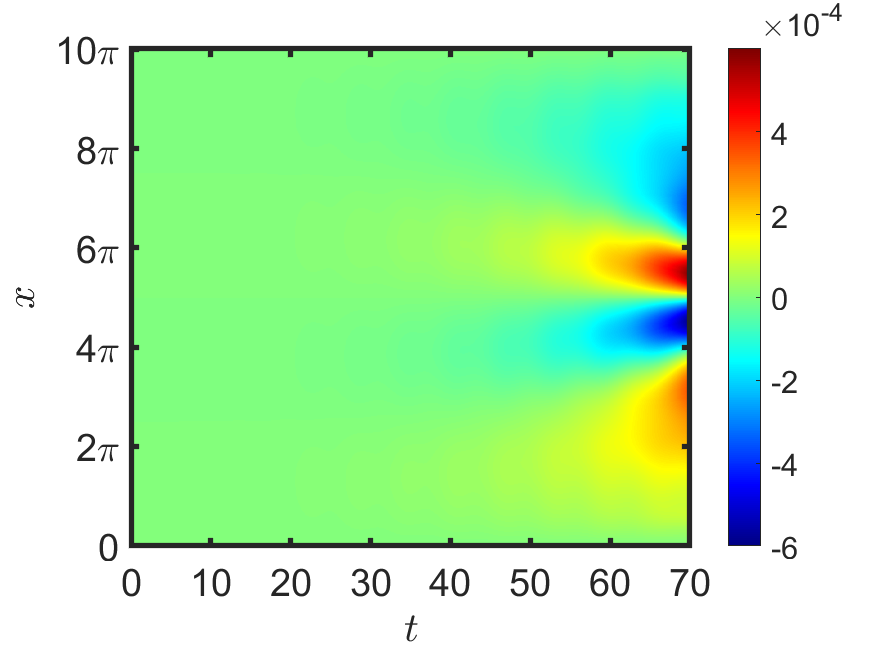}
		\label{fig:weibel_sol_E1}
	} \quad
	\subfigure[$E_1$, DVM]{
		\includegraphics[width=0.45\linewidth]{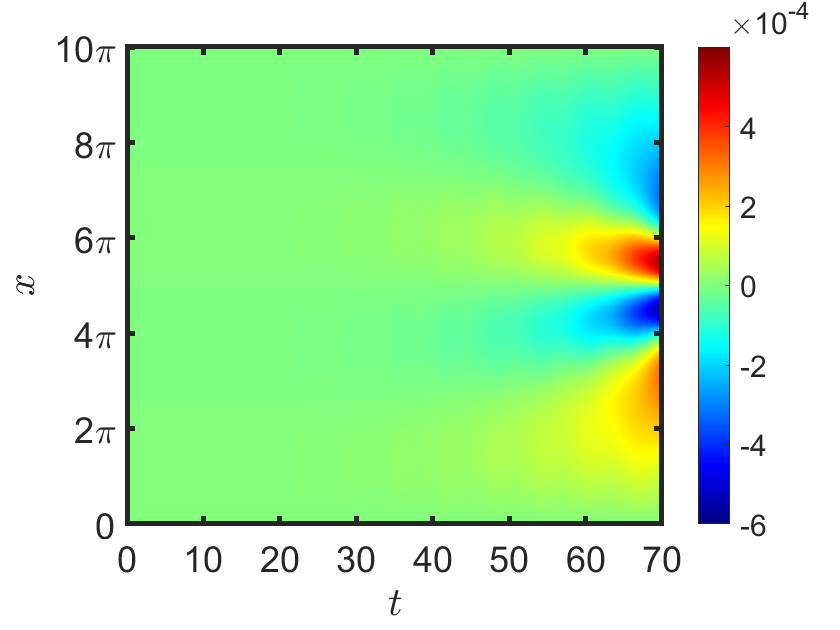}
		\label{fig:weibel_sol_E1_DVM}
	}
	\subfigure[$E_2$, moment]{
		\includegraphics[width=0.45\linewidth]{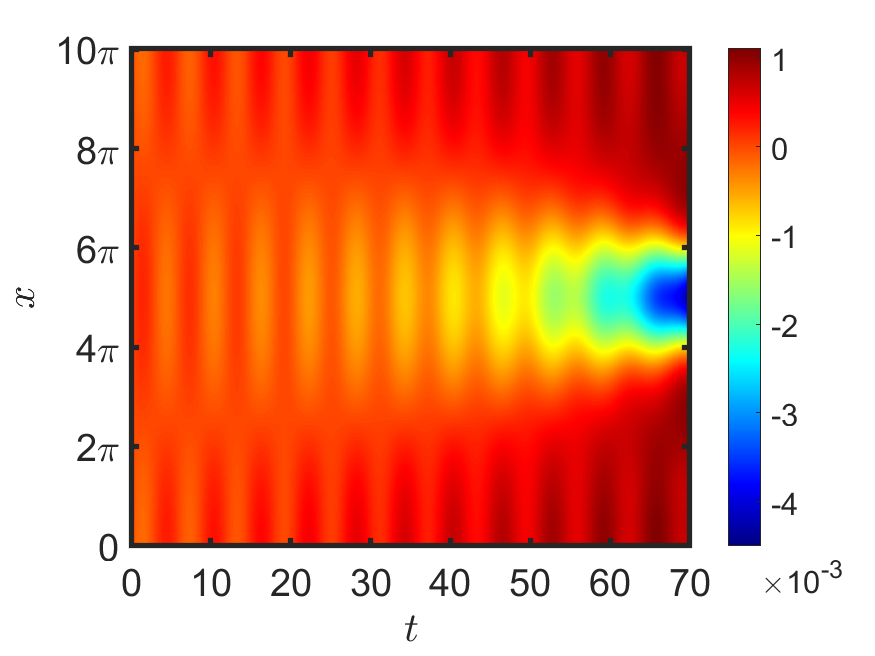}
		\label{fig:weibel_sol_E2}
	} \quad
		\subfigure[$E_2$, DVM]{
		\includegraphics[width=0.45\linewidth]{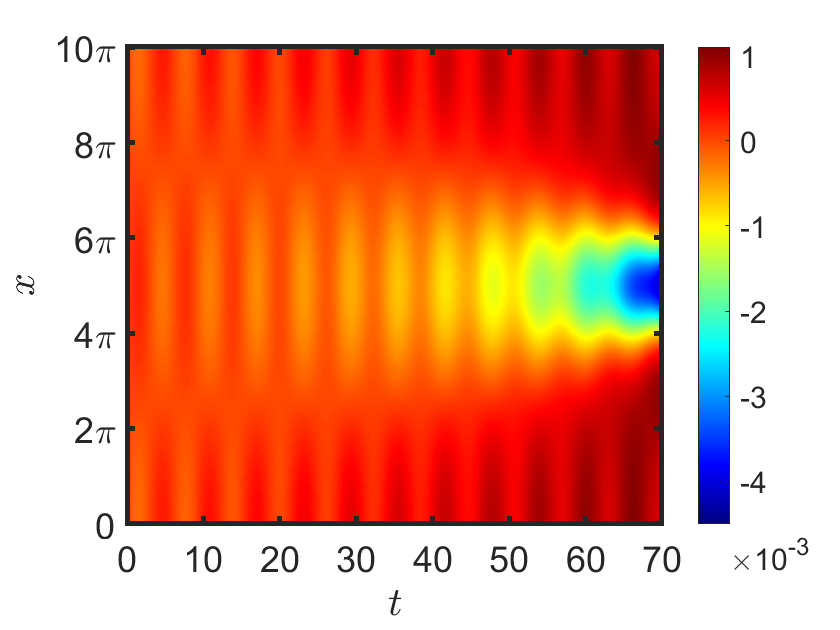}
		\label{fig:weibel_sol_E2_DVM}
	} 
	
	\subfigure[$B_3$, moment]{
		\includegraphics[width=0.45\linewidth]{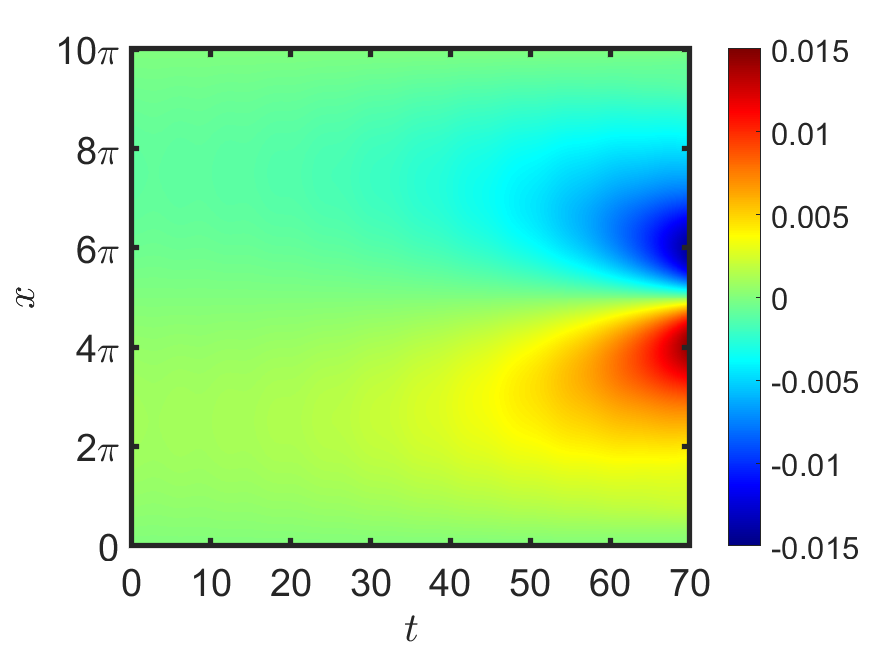}
		\label{fig:weibel_sol_B3}
	}\quad
	\subfigure[$B_3$, DVM]{
		\includegraphics[width=0.45\linewidth]{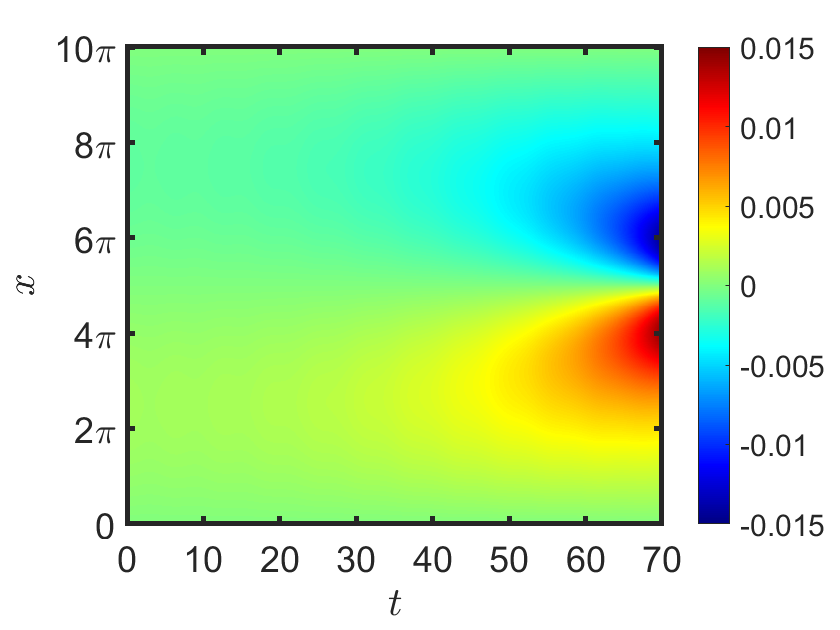}
		\label{fig:weibel_sol_B3_DVM}
	}
	\caption{Weibel instability in Sec. \ref{sec:weibel}.
	 Time evolution electric fields $E_1$ (top) and $E_2$ (middle), and magnetic field $B_3$ (bottom), obtained by the moment method (left) and the DVM (right). 
	 }
	\label{fig:weibel_EB}
\end{figure}


Fig. \ref{fig:weibel_sol} shows the time evolution of kinetic, magnetic, and electric energies. Two components of the electric energy are also plotted. It can be found that 
after the fast transient, both the magnetic and electric energy are increasing with oscillations. It can be observed that the electric energy $E_1$ in the $x$-direction is quite small compared with $E_2$ in the $y$-direction, which is unanimous with the setting of the initial condition. Fig. \ref{fig:weibel_sol_energy_N} shows the numerical solutions of $\mE_{K,h}^n, \mE_{E,h}^h, \mE_{B, h}^n$ and $\mE_{E_1,h}^n$ with different spatial sizes $N = 400, 800, 1000$ and $2000$. Here we do not show $\mE_{E_2, h}^n$, since $\mE_{E_1, h}^n$ is very small compared with $\mE_{E_2, h}^n$, thus $\mE_{E_2, h}^n$ almost equals $\mE_{E,h}^h$. For different spatial sizes, the numerical solutions are on top of each other, which means that with $N = 400$, the numerical solution can resolve the evolution of each energy well. 
In Fig. \ref{fig:weibel_sol_energy}, the reference solution by the DVM is also plotted and the numerical solution is consistent with the reference solution.

The marginal distribution function $g(t, x, v_2)$ defined as
\begin{equation}
  \label{eq:MDF2}
  g(t, x, v_2) = \int_{\bbR} f(t, x, \bv) \dd v_1.
\end{equation}
at $t = 0, 30, 50$ and $70$ is plotted in Fig. \ref{fig:weibel_sol_dis}. At the initial moment, there are a high peak and a low peak, and the oscillations appear gradually. At $t = 70$, there exists an obvious sink in the middle of the distribution function. The time evolution of the electromagnetic field compared to the reference solution by the DVM is shown in Fig. \ref{fig:weibel_EB}, where the magnitude of $E_1$ is the smallest, which is consistent with the numerical result in Fig. \ref{fig:weibel_sol}. We clearly observe the the oscillations in $E_2$, the periodic structure in $E_1$ and $B_3$.


\subsection{Orszag-Tang vortex}
\label{sec:Orszag-Tang}
In this section, we consider the Orszag-Tang vortex problem, which is a classic example in magnetohydrodynamics (MHD) \cite{parashar2010orszag}. It is an example of the interaction between large-scale fluid behavior and small-scale dissipation processes involving dynamic physics, which still interests a lot of research nowadays. 
The Orszag-Tang vortex problem describes the time evolution of ions and fluid electrons, and its initial condition rapidly leads to broadband turbulence. Here, the Orszag-Tang vortex problem is described using the multi-species VM system in the 2D3V setting. We study this problem with the energy-preserving moment method. We refer readers to \cite{cheng_numerical_2015, cheng_energy-conserving_2015} and the references therein for energy-conserving schemes of two-species VA systems. We first introduce the multi-species VM system. 

\paragraph{Multi-species VM system}
In a collisionless magnetized plasma, the normalized time evolution equation of the VM model with $s$ species has the form below. The governing equation for the $k$-th species is
\begin{equation} 
\label{eq:multi_particle}
	\dfrac{\partial f^k}{\partial t}+\nabla_{\bm{x}} \cdot (\bm{v}f^k)+ \frac{q^k}{m^k}\frac{\omega_{ce}}{\omega_{pe}}(\bm{E}+\bm{v}\times\bm{B})\cdot\nabla_{\bm{v}}f^k=0, \qquad k = 1, 2, \cdots, s. 
\end{equation}
Similarly, the normalized Maxwell's equations describing the electromagnetic field have the form 
\begin{equation}
\label{eq:mul_Maxwell}
		\left\{
		\begin{aligned}
			&\frac{\partial\bm{E}}{\partial t}-\nabla_{\bm{x}}\times \bm{B}=-\frac{\omega_{\rm pe}}{\omega_{\rm ce}}\bm{J},\\
			&\frac{\partial \bm{B}}{\partial t}+\nabla_{\bm{x}}\times\bm{E}=0,
		\end{aligned}
		\right.
\end{equation}
with Gauss's law as 
\begin{equation}
\label{eq:multi_gauss_law}
	\nabla_{\bm{x}}\cdot\bm{E}=\frac{\omega_{\rm pe}}{\omega_{\rm ce}}\rho_{\rm free},\qquad\nabla_{\bm{x}}\cdot\bm{B}=0.
\end{equation}
Here, $q^k$ and $m^k$ are the normalized charge and mass of the $k$th species of particles. $\omega_{\rm ce}$ and $\omega_{\rm pe}$ are the 
electron cyclotron frequency and the electron plasma frequency. In the numerical simulation, we pay attention to the ratio $\omega_{\rm ce} / \omega_{\rm pe}$ instead of their individual values. Moreover, the mass density $\rho_{m}(t, \bx)$ and electric current $J(t, \bx)$ for the multi-species VM model are defined as 
\begin{equation}
\begin{aligned}
\label{eq:mul_rho_J}
	&\rho_{\text{m}}(t,\bm{x}) = \sum_{k=1}^s \rho_{m}^k = \sum_{k=1}^s m^k\rho^k = \sum_{k=1}^s m^k\int_{\mathbb{R}^3}f^k(t,\bm{x},\bm{v})\dd\bm{v},
	\\
	&\bm{J}(t,\bm{x})=\sum_{k=1}^s \bm{J}^k =\sum_{k=1}^s q^k \rho^k\bm{u}^k =  \sum_{k=1}^sq^k\int_{\mathbb{R}^3}\bm{v}f^k(t,\bm{x},\bm{v})\dd\bm{v},
	\end{aligned}
\end{equation}
where $\rho_m^k$  and $\rho^k$ are the mass and number density of the $k$-th species of particles.
The momentum and temperature of the $k$-th particle are defined as 
\begin{equation}
  \label{eq:mul_u_T}
  \rho_m^k\bu^k = m^k \int_{\bbR^3} \bv f^k \dd \bv, \qquad 
  \rho^k |\bu^k|^2 + 3 \rho^k \mT^k = \int_{\bbR^3} |\bv|^2 f^k \dd \bv. 
\end{equation}
In this multi-species model, the discrete total energy, electric energy, magnetic energy, and kinetic energy of the system at time $t^n$ are defined as
\begin{equation}
  \label{eq:multi_energy}
 	\mathcal{E}^n_{{\rm total}, h}=\sum_{k=1}^s\mathcal{E}^{n, k}_{K, h} +\mathcal{E}^n_{E, h} + \mathcal{E}^n_{B, h},
\end{equation} 
with 
\begin{equation}
\label{eq:multi_E_B_K}
\begin{aligned}
 	& \mathcal{E}^n_{E, h} =\frac{1}{2}\left(\frac{\omega_{ce}}{\omega_{pe}}\right)^2\sum_{\bm{j}\in\bm{J}}\Delta V |\bE^n_{\bj}|^2, \qquad 
 	\mathcal{E}^n_{B, h} =\frac{1}{2}\left(\frac{\omega_{ce}}{\omega_{pe}}\right)^2\sum_{\bm{j}\in\bm{J}} \Delta V \bB^{n+1/2}_{\bj} \cdot \bB^{n-1/2}_{\bj} , \\
 &	\mathcal{E}^{n, k}_{K, h} =\frac{m^k}{2}\sum_{\bm{j}\in\bm{J}}\Delta V
	(\rho^k_{\bj})^n\big[|\bu^n_{\bj}|^2 + 3(\mathcal{T}^k_{\bj})^n\big].
	\end{aligned}
\end{equation}

\begin{remark}
The normalization of the multi-species VM system \eqref{eq:multi_particle}, \eqref{eq:mul_Maxwell} and \eqref{eq:multi_gauss_law} is similar to \cite{Multi_hermite}. The elementary charge $e$ and the electron mass $m^e$ are treated as standard charge and mass, respectively. $\epsilon_0$ is the permittivity of the vacuum, while $n_0$ is a reference electron density. Then, the normalization is done as below
\begin{equation}
  \label{eq:normalization}
  \hat{t} = \frac{t}{1/\omega_{\rm pe}}, \qquad \hat{\bv} = \frac{\bv}{c}, \qquad \hat{\bx} = \frac{\bx}{L_0}, \qquad \hat{\bB} = \frac{\bB}{B_0}, \qquad \hat{\bE} = \frac{\bE}{c B_0},
\end{equation}
where $\omega_{\rm pe}=\sqrt{\frac{e^2n_0^e}{\epsilon_0 m^e}}$, and $c$ is the speed of light. The inertial length of the electron $L_0 = c / \omega_{\rm pe}$. Finally, we can derive the cyclotron frequency of species $k$ as $\omega_{\rm ck}=eB_0/m^k$.
\end{remark}

For the Orszag-Tang vortex problem, the initial condition is as follows, and we refer readers to \cite{Multi_hermite, UGKWPVM} for more details. 
\begin{align}
\label{eq:ex4_ini}
	f^k(0,\bm{x},\bm{v})=\frac{\rho^k}{(2\pi \mathcal{T}^k)^{3/2}}
	\exp\left(-\frac{|\bm{v}-\bm{U}^k|^2}{2\mathcal{T}^k}\right),\qquad (\bm{x},\bm{v})\in[0,L]^2\times\mathbb{R}^3, \qquad k = i, e,
\end{align}
with 
\begin{equation}
\begin{split}
	&\rho^k = \gamma^2, \qquad \bm{U}^k = [-\bar{B} \bar{v} \sin(y), \bar{B} \bar{v} \sin(x), 0], \qquad k = i, e,\\
	&\bE = \bz, \qquad \bB = [0, \bar{B} \sin(2x), 0], \qquad \mT^i = 0.024, \qquad \mT^e = 0.6,
\end{split}
\end{equation}
where 
\begin{equation}
  \label{eq:ex4_ini_value}
  	\frac{m^i}{m^e}=25,\qquad\frac{\omega_{\rm pe}}{\omega_{\rm ce}} = 1, \qquad \gamma = 5/3, \qquad L = 2 \pi, \qquad \bar{B} = 0.5, \qquad \bar{v} = 0.5. 
\end{equation}
Moreover, the pressure $P^k$ is defined as 
\begin{equation}
  \label{eq:OT_P}
  P^k = \rho_{\text{m}}^k \mT^k=m^k \rho^k \mT^k, \qquad k = i, e.
\end{equation}
\begin{figure}[!ht]
	\centering
	\subfigure[time evolution of the energy]{
		\includegraphics[width=0.45\linewidth]{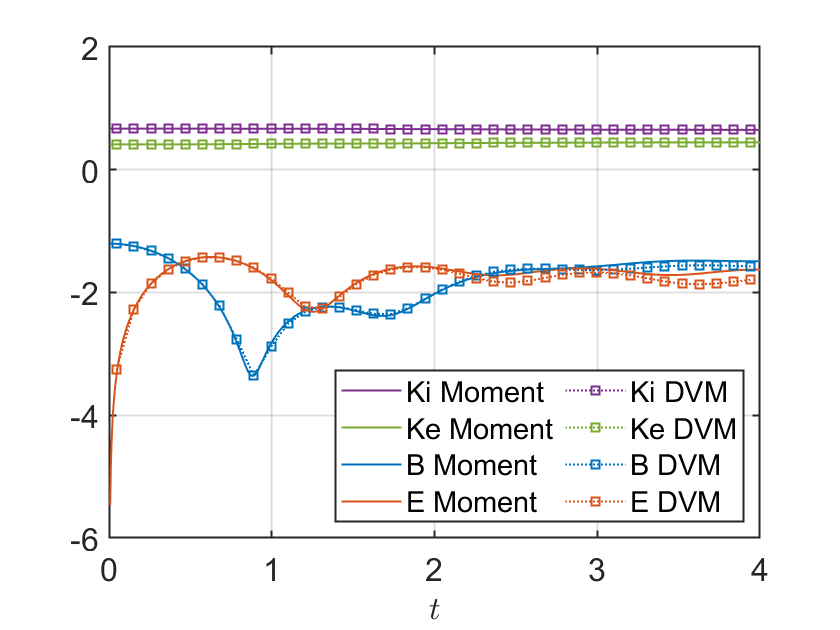}
		\label{fig:OT_energy}
	}
	\subfigure[relative error of the total energy]{
		\includegraphics[width=0.45\linewidth]{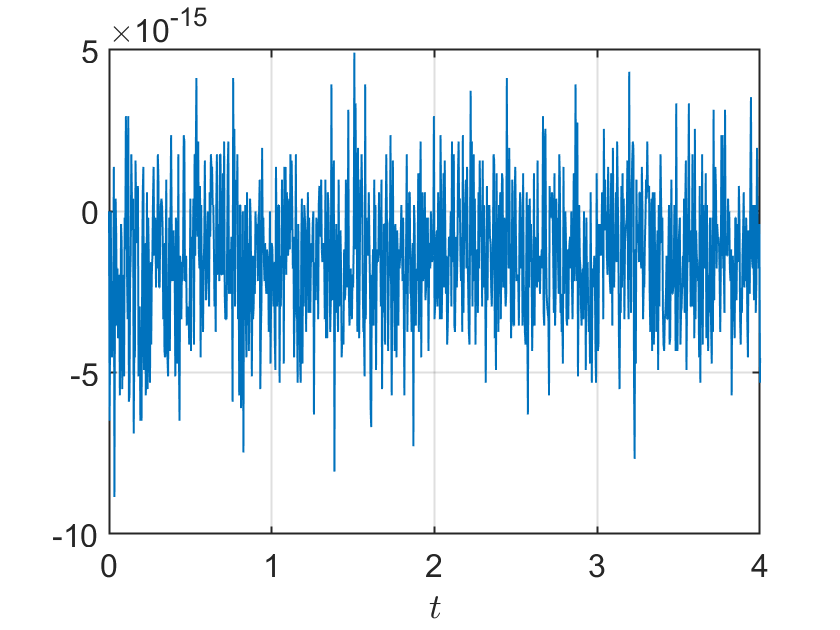}
		\label{fig:OT_var_energy}
	}
	\caption{Orszag-Tang vortex in Sec. \ref{sec:Orszag-Tang}. (a) Time evolution of the electromagnetic field energy and kinetic energy of ions and electrons defined in \eqref{eq:multi_E_B_K}. 
	$y$-axis denotes logarithmic form of the energies $\log_{10}(\mathcal{S}/L^2)$ with $\mathcal{S} = \mE^n_{B,h}, \mE^n_{E,h}, (\mE^k_{K,h})^n, k=i,e$.
	(b) Time evolution of the relative error for the total energy as $\mV(\mE^n_{{\rm total},h})$ defined as \eqref{eq:var_energy}. 
	$y$-axis denotes $\mV(\mE^n_{{\rm total},h})$.}
	\label{fig:Orszag_energy}
\end{figure}

For this problem in 2D3V setting, the numerical simulations are quite expensive and thus are performed with a relatively coarser mesh $N_x=N_y=100$ in the spatial domain, and with a larger CFL as $\text{CFL}=0.4$. The truncation order of the moment method is $M = 30$.
For the MHD problem, people are more interested in the multiples of Alfv\'{e}n time defined as $t_{\text{Alfv\'{e}n}}=L_0/\bar{v}$, where $L_0$ equals $1$ in the normalized system \eqref{eq:multi_particle}, see e.g. \cite{kigure2010generation} for more details. In this test, the final simulation time is set as $t = 2t_{\text{Alfv\'{e}n}}=4$. The time evolution of the electric energy $\mathcal{E}^n_{E, h}$, magnetic energy $\mathcal{E}^n_{B, h}$, and kinetic energy $\mE^{n, s}_{K, h}, s= i, e$ are shown in Fig. \ref{fig:OT_energy}, where the reference solution obtained by the DVM is also plotted. 
It can be observed that for the evolution of each energy, the numerical solution fits well with the reference solution. Fig. \ref{fig:OT_var_energy} illustrates the time evolution of the relative error for the total energy. We can find that the energy-conserving moment method can preserve the total energy for this complicated multi-species system in the 2D3V setting. 

\begin{figure}[!htbp]
	\centering
	\includegraphics[width=1.0\linewidth]{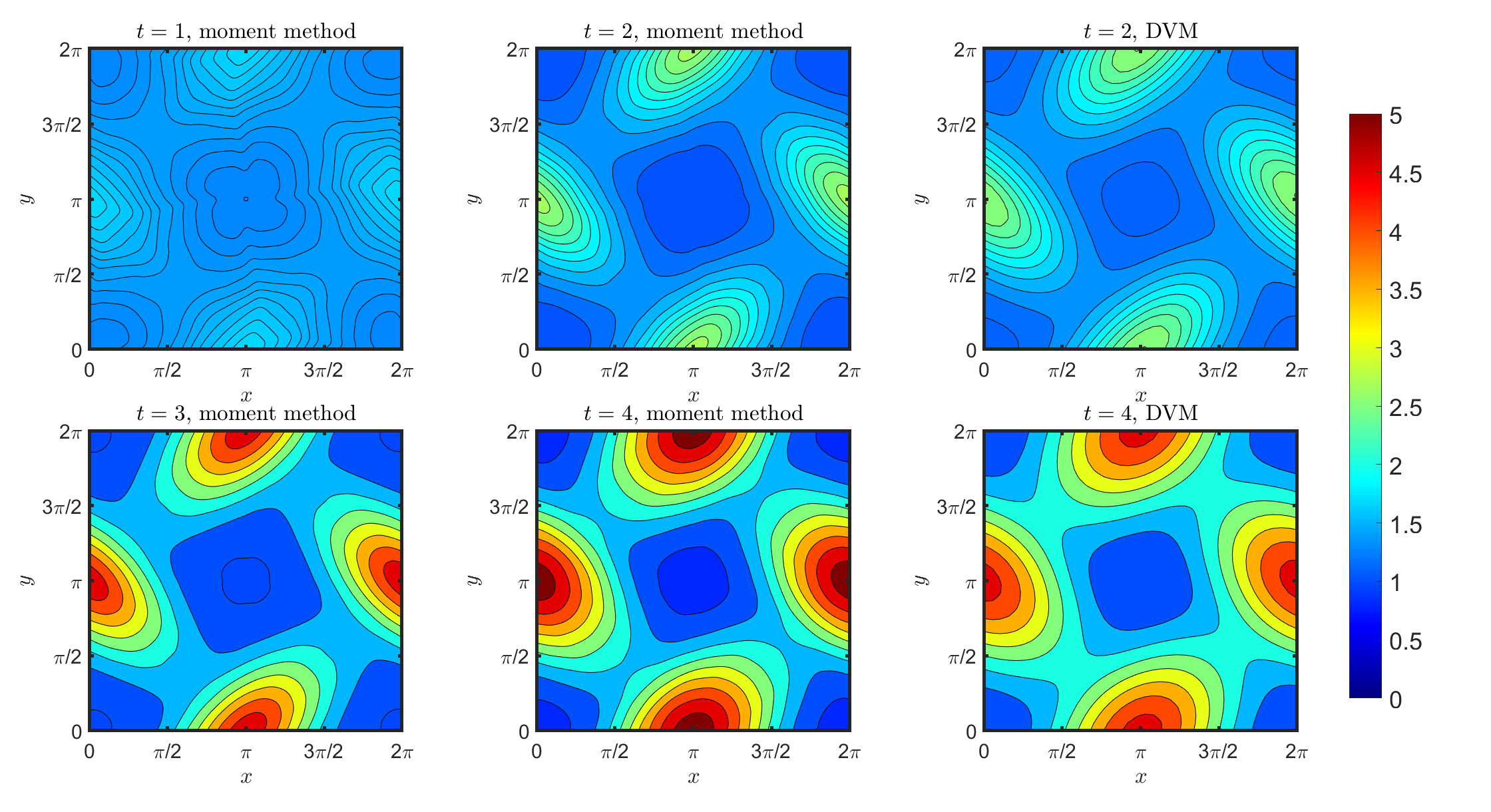}
	\caption{Orszag-Tang vortex in Sec. \ref{sec:Orszag-Tang}. Contour plots of the pressure for the ion $P^i=\rho_m^i\mathcal{T}^i$ at $t=1, 2, 3$ and $4$. The left and middle columns are numerical solutions by the moment method and the right column is the reference solution by the DVM.}
	\label{fig:OT_P}
\end{figure}

Fig. \ref{fig:OT_P} shows the pressure distribution diagram of the ions at different time, and the reference pressure at $t = 2$ and $4$ is also plotted. At first, the pressure is uniform and small vortexes start to form gradually. At $t = 2t_{\text{Alfv\'{e}n}} = 4$, there are nearly four vortexes, and the numerical results are consistent with the reference solution. In the Orszag-Tang vortex problem, people are always interested in the density $\rho_m$ and the current $J_z$. The mass density $\rho_{m}^i$ and the current of the ions $J_z^i = q^i \rho^i u_z^i$ at time $t = 1, 2, 3$ and $4$ are illustrated in Fig. \ref{fig:Orszag_rho} and \ref{fig:Orszag_Jz}, respectively. The evolution of $ \rho^i_{m}$ is similar to that of pressure $P^i$, which is uniform at the beginning, and four vertices form at $t = 4$. For the current, it is also smooth at the beginning, and evolves to several vortexes and oscillations as time goes.

The time evolution of the magnetic field is shown in Fig. \ref{fig:OrszagB}, where the same color interval is utilized for all figures. The background is the total magnetic energy, and the white lines with arrows are the magnetic fields. At the initial time as in Fig. \ref{fig:OrszagB0}, the magnetic field has four peaks, forming four uniform magnetic field bands in the $xy$-plane. Then, the peak value gradually decreases, but the direction of the magnetic field is barely changed in Fig. \ref{fig:OrszagB5}. In Fig. \ref{fig:OrszagB10} and 
\ref{fig:OrszagB15}, the total magnetic energy becomes quite small, and each band is distorted, with two peaks forming. Then, two peaks at the center of the $xy$-plane and four peaks at the four corners appear in Fig. \ref{fig:OrszagB20} and \ref{fig:OrszagB25}. Finally, the magnetic field lines begin to twist in Fig. \ref{fig:OrszagB30}. 
At $t = 4$, the double peaks in the center merge to one peak, and the five peaks are growing larger and larger in Fig. \ref{fig:OrszagB40}. 
From Fig. \ref{fig:OrszagB}, we can clearly see the evolution of the magnetic field, and more complicated phenomenon may appear with time going, which will be left for future work.

 \begin{figure}[!htbp]
 	\centering
 	\includegraphics[width=1.0\linewidth]{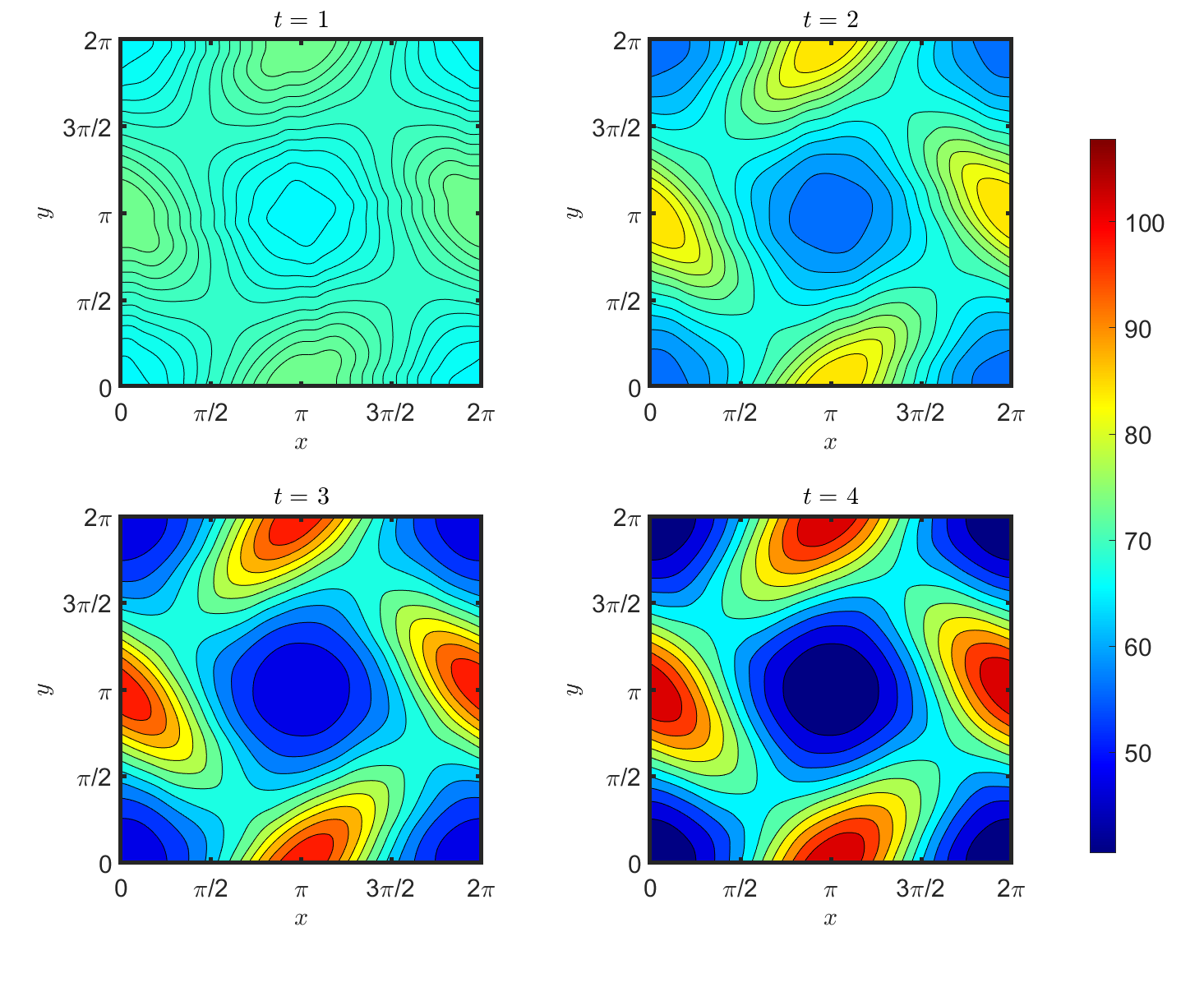}
 	\caption{Orszag-Tang vortex in Sec. \ref{sec:Orszag-Tang}. Contour plots of the mass density of the ions $\rho_{m}^i$ at $t = 1, 2, 3$ and $4$. }
 	\label{fig:Orszag_rho}
 \end{figure}
 
 \begin{figure}[!htbp]
 	\centering
 	\includegraphics[width=1.0\linewidth]{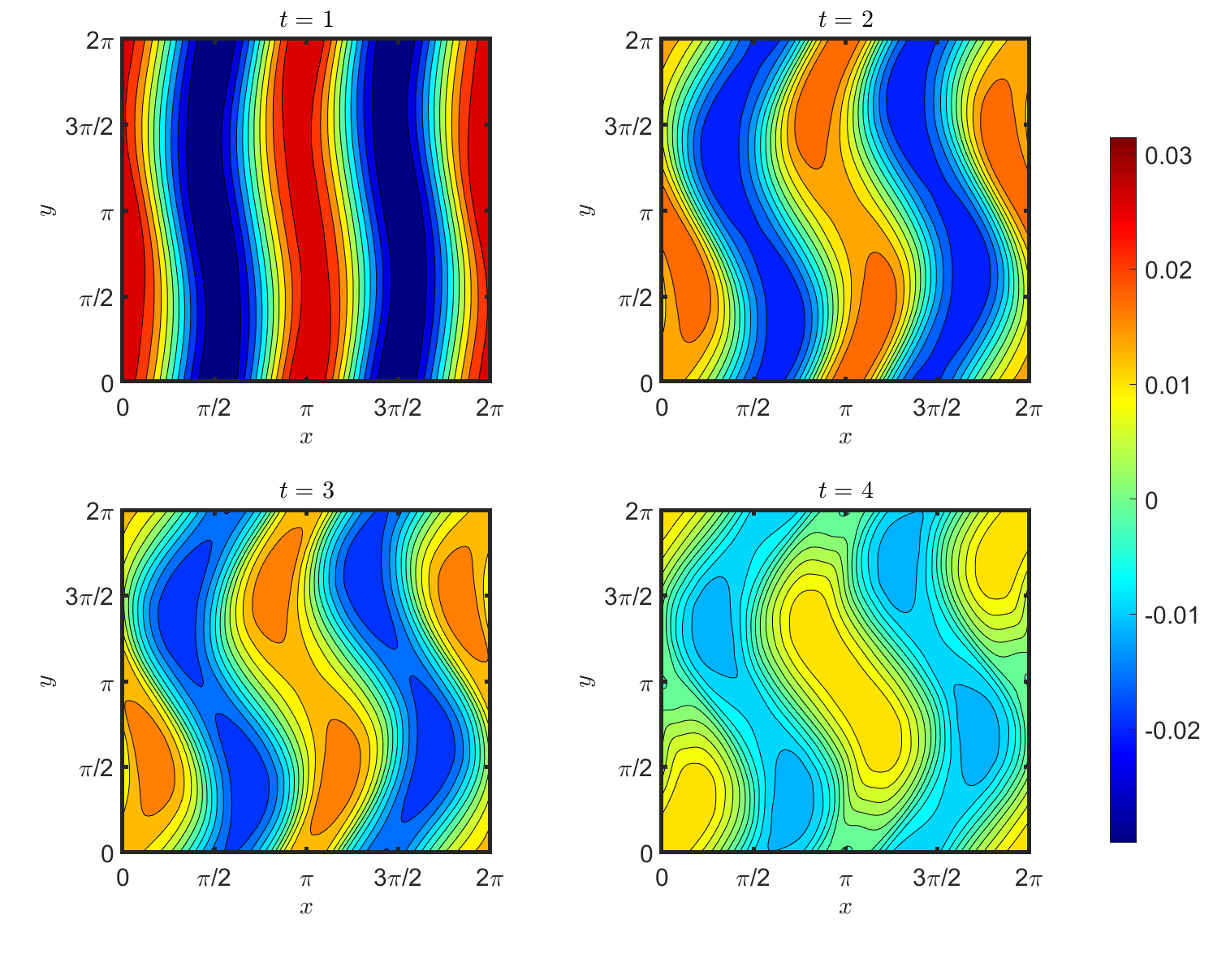}
 	\caption{Orszag-Tang vortex in Sec. \ref{sec:Orszag-Tang}. Contour plots of the current density of ions $J_z^i$ at $t = 1, 2, 3$ and $4$. }
 	\label{fig:Orszag_Jz} 
\end{figure}
 
 \begin{figure}[!ht]
	\centering
	\subfigure[$t=0$]{
	\includegraphics[width=0.32\linewidth]{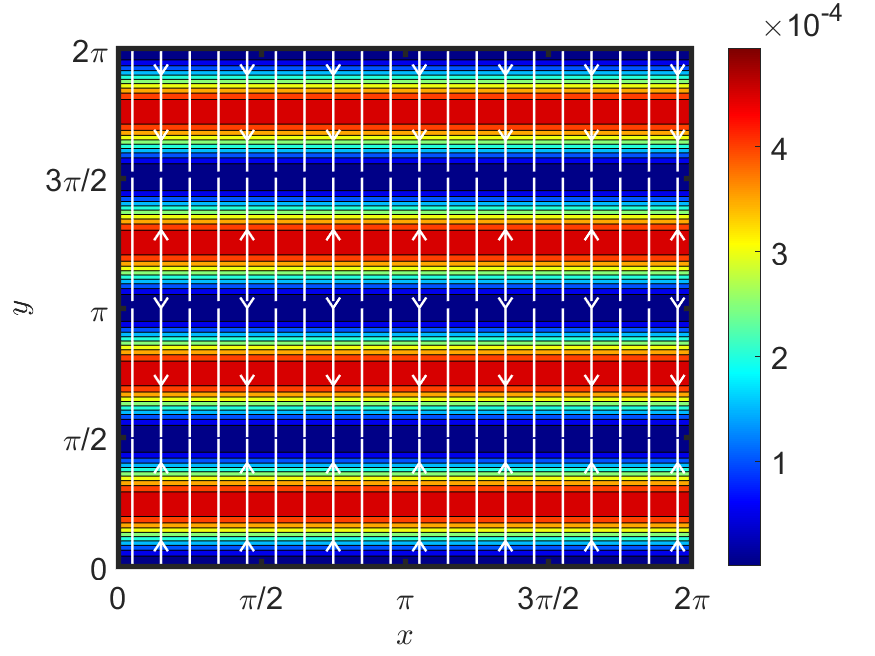}
	\label{fig:OrszagB0}
	}\hspace{-0.3cm}
	\subfigure[$t=0.5$]{
	\includegraphics[width=0.32\linewidth]{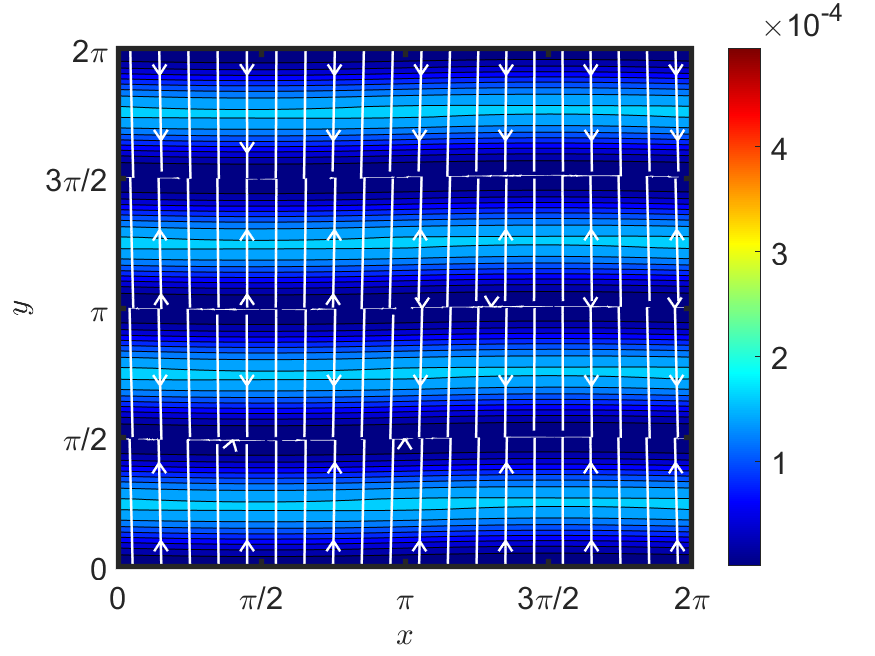}
		\label{fig:OrszagB5}
	}\hspace{-0.3cm}
	\subfigure[$t=1$]{
	\includegraphics[width=0.32\linewidth]{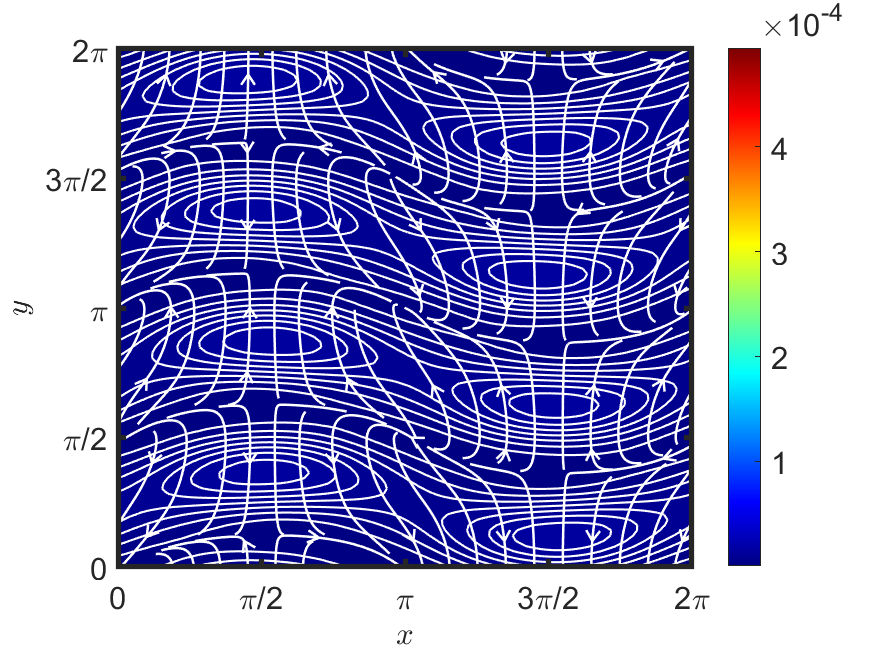}
		\label{fig:OrszagB10}
	}\\
	\centering
	\subfigure[$t=1.5$]{
	\includegraphics[width=0.32\linewidth]{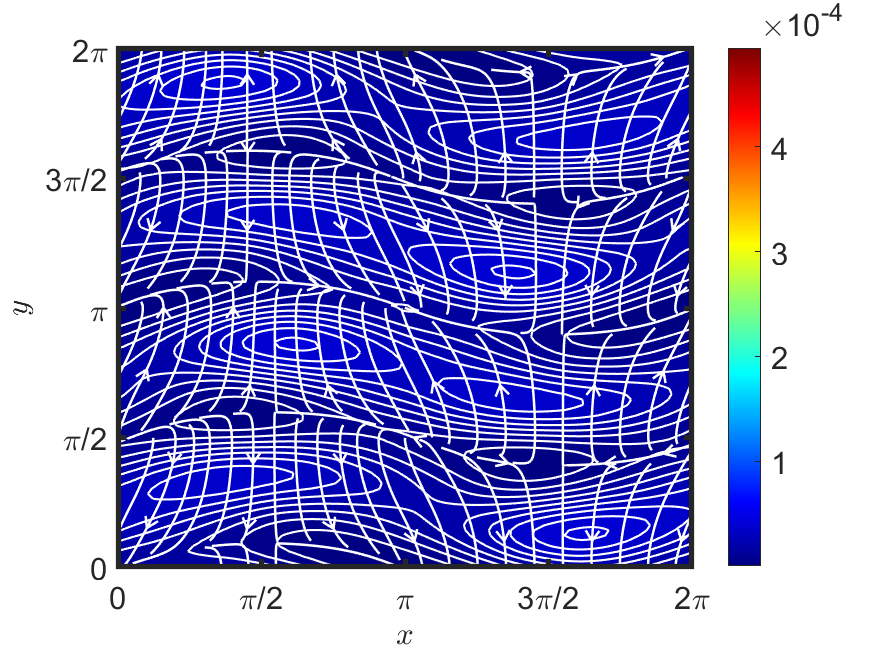}
	\label{fig:OrszagB15}
	}\hspace{-0.3cm}
	\subfigure[$t=2$]{
	\includegraphics[width=0.32\linewidth]{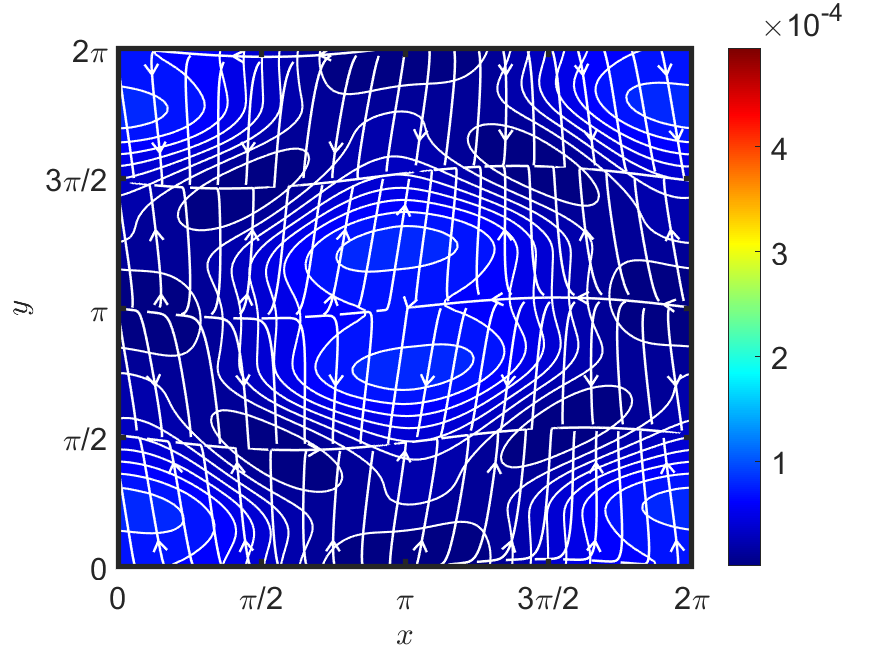}
	\label{fig:OrszagB20}
	}\hspace{-0.3cm}
	\subfigure[$t=2.5$]{
	\includegraphics[width=0.32\linewidth]{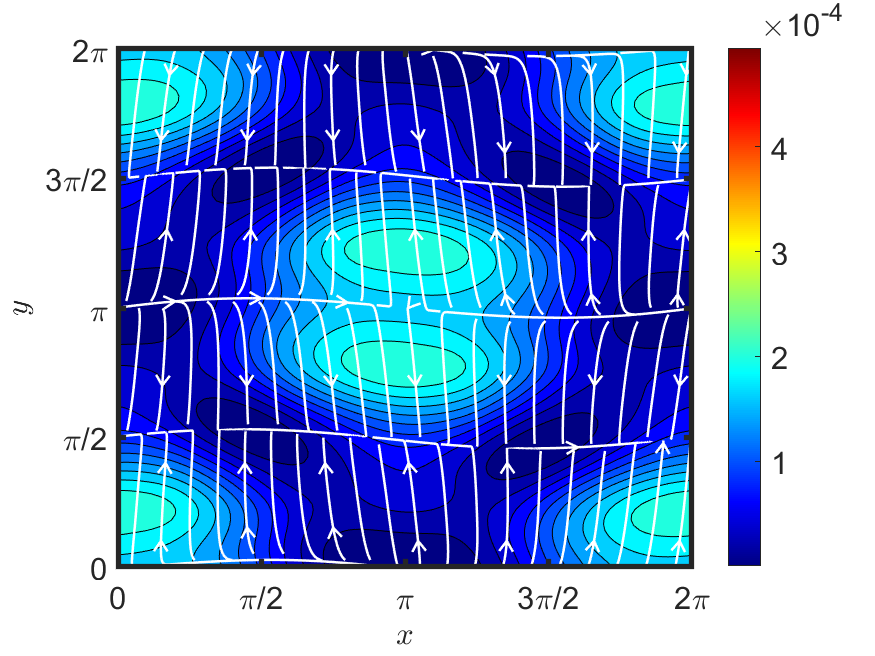}
	\label{fig:OrszagB25}
	}\\
	\centering
	\subfigure[$t=3$]{
	\includegraphics[width=0.32\linewidth]{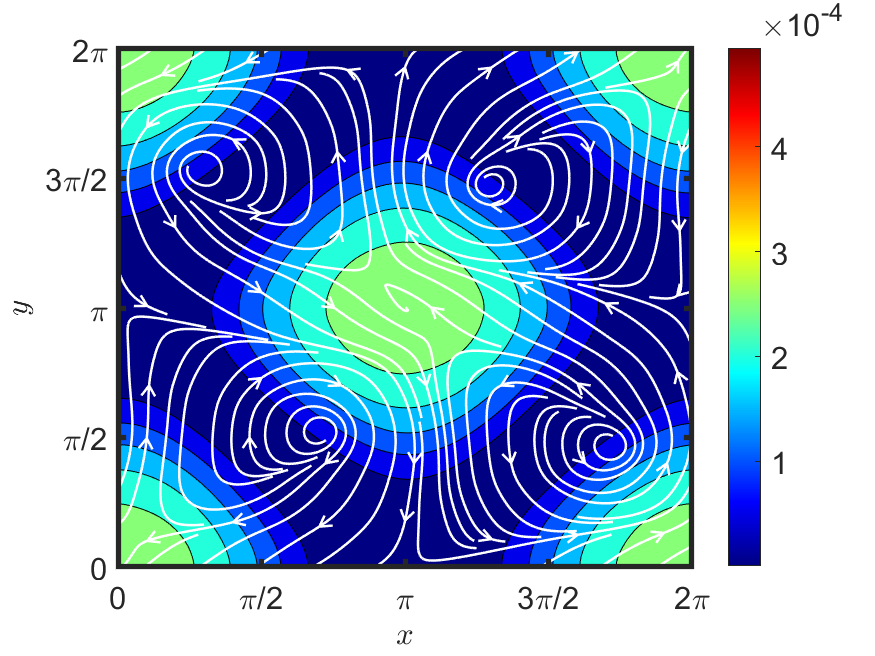}
	\label{fig:OrszagB30}
	} \hspace{-0.3cm}
	\subfigure[$t=3.5$]{
	\includegraphics[width=0.32\linewidth]{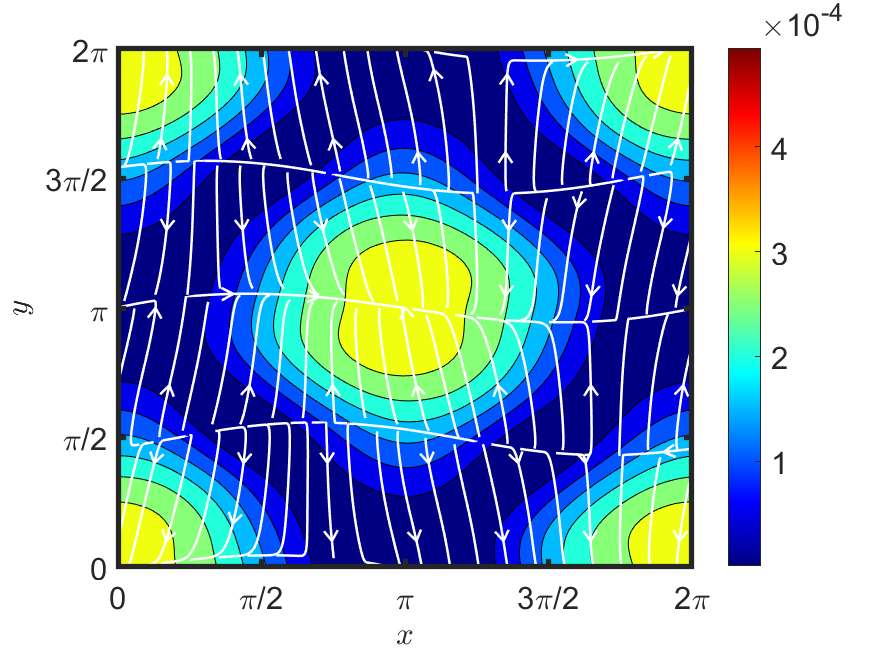}
	\label{fig:OrszagB35}
	}\hspace{-0.3cm}
	\subfigure[$t=4$]{
	\includegraphics[width=0.32\linewidth]{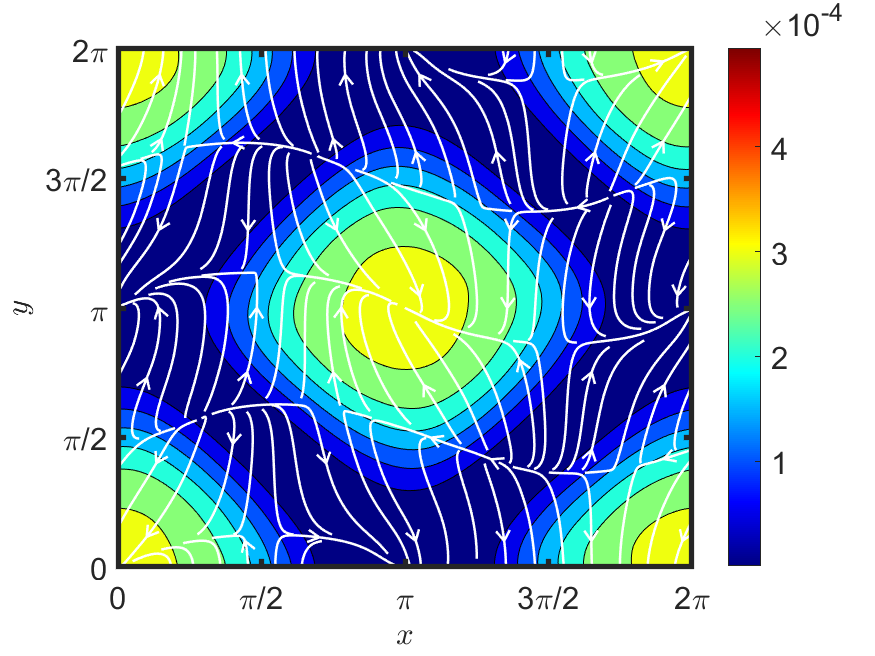}
	\label{fig:OrszagB40}
	}\\
	\caption{Orszag-Tang vortex in Sec. \ref{sec:Orszag-Tang}. Time evolution of the magnetic field. Here, the background is the total magnetic energy, and the white line is the magnetic field.}
	\label{fig:OrszagB}
\end{figure}


\section{Conclusion}
\label{sec:conclusion}
In this paper, the globally hyperbolic moment system is derived for the VM system under the framework of the Hermite spectral method. With the expansion center chosen as the local macroscopic velocity and the scaling factor as the temperature, the movement of the particles led by the Lorentz force can be expressed with the linear combination of the moment coefficients. Therefore, only quite a few degrees of freedom are needed to describe the effect of the Lorentz force. An energy-preserving numerical scheme for the moment system, where only a small linear equation system needs to be solved, is proposed. The numerical results of Landau damping, two-stream instability, Weibel instability, and Orszag-Tang vortex problem are shown to verify the efficiency and energy-preserving ability of this numerical scheme. 

It is illustrated that the globally hyperbolic moment system is a promising alternative to model the VM system. The numerical scheme with large time step and more applications will be studied in the future.

\section*{Acknowledgements}
 Xinghui Zhong is partially supported by the National Natural Science Foundation of China (Grant No. 11871428). Yanli Wang is partially supported by the National Natural Science Foundation of China (Grant No. 12171026, U1930402 and 12031013).
 The authors would like to thank Prof. Ruo Li from Peking University and Prof. Zhenning Cai from National University of Singapore for
their valuable suggestions. 

\bibliographystyle{plain}
\bibliography{Ref.bib}
\end{document}